\documentclass[11pt,a4paper,oneside,reqno]{amsart}
\usepackage[utf8]{inputenc}
\usepackage[T1]{fontenc}
\usepackage{soul}
\usepackage{amsfonts, amsthm, amsmath}
\usepackage{mathrsfs}

\usepackage{graphicx}
\usepackage{amsmath}
\usepackage{ amssymb}
\usepackage{amsfonts}
\usepackage{amsthm}
\usepackage{enumitem}
\usepackage{subcaption}
\usepackage{mathdots}
\usepackage{enumitem}
\usepackage{times}
\usepackage{mathrsfs}
\usepackage{multirow}
\usepackage{comment}

\newtheorem{theorem}{Theorem}[section]
\newtheorem{proposition}{Proposition}[section]
\newtheorem{corollary}[theorem]{Corollary}
\newtheorem{lemma}[theorem]{Lemma}
\theoremstyle{definition}
\newtheorem{example}{Example}[theorem]
\newtheorem{remark}[theorem]{Remark}
\newtheorem{definition}[theorem]{Definition}

\usepackage{xfrac}
\usepackage{blindtext}
\usepackage{mathtools}
\usepackage{array}
\usepackage{multirow} 
\usepackage{booktabs}
\usepackage{tabularx}
\usepackage{makecell}
\usepackage{flafter}
\usepackage{amsaddr}

\newcommand{\C}{\mathbb{C}}
\newcommand{\R}{\mathbb{R}}
\newcommand{\mc}{\mathcal}
\newcommand{\N}{\mathbb{N}}

\DeclareMathOperator{\diag}{diag}
\DeclareMathOperator{\lcm}{lcm}
\DeclareMathOperator{\mult}{mult}
\DeclareMathOperator{\sspan}{span}
\newcommand{\trans}{^\top}
\newcommand{\simP}{\stackrel P\sim}

\DeclareMathOperator{\P3}

\newcommand{\npmatrix}[1]{\left( \begin{matrix} #1 \end{matrix} \right)}

\title{Arbitrarily Finely Divisible Matrices}

\begin{document}

\maketitle

\begin{center}
Priyanka Joshi\textsuperscript{1}, Helena \v Smigoc\textsuperscript{1} \par \bigskip
 \textsuperscript{1}School of Mathematics and Statistics, University College Dublin, Belfield, Dublin 4, Ireland \par
\end{center}

\begin{abstract}
The class of stochastic matrices that have a stochastic $c$-th root for infinitely many natural numbers $c$ is introduced and studied. Such matrices are called arbitrarily finely divisible, and generalise the class of infinitely divisible matrices. In particular, if $A$ is a transition matrix for a Markov process over some time period, then arbitrarily finely divisibility of $A$ is the necessary and sufficient condition for the existence of transition matrices corresponding to this Markov process over arbitrarily short periods.  

In this paper, we lay the foundation for research into arbitrarily finely divisible matrices and demonstrate the concepts using specific examples of $2 \times 2$ matrices, $3 \times 3$  circulant matrices, and rank-two matrices.



\end{abstract}

\section{Introduction}

A discrete-time Markov process can be described by its transition matrix, a stochastic matrix $A$ whose $(i,j)$-th element is the probability of moving from state $i$ to state $j$ over a time interval. In practice, a transition matrix may be estimated for a certain time period $t_0$, however, a
transition matrix for a period shorter than $t_0$ may be needed. If $A$ is a transition matrix for a time period $t_0$, then a stochastic $c$-th root of $A$ is a transition matrix for a time period $\frac{t_0}{c}$. In particular, if a transition matrix $A$ possesses a stochastic $c$-th root for every $c\in \N$, then transition matrices for arbitrarily short time intervals can be found. Such matrices $A$ are called \emph{infinitely divisible}, \cite{vanbrunt}. If in addition $\det(A)>0$ then $A$ is said to be \emph{strongly infinitely divisible}. It was proved by Kingman \cite{kingman} that if a nonsingular stochastic matrix $A$ has a $c$-th root for every $c\in \N$, then $A=e^Q$, where $Q$ has row sums zero and nonnegative off-diagonal entries. 
In this case, $A$ is called \emph{embeddable}, and the matrix $Q$ is said to be the generator of $A$.  Building on Kingman \cite{kingman}, 
 Van Brunt \cite{vanbrunt} established that an $n\times n$ nonnegative matrix $A$ is strongly infinitely divisible if and only if there exists a matrix $Q$ with nonnegative off-diagonal entries such that $A=e^{Q}.$ 

Deciding if a given stochastic matrix is embeddable is known as \emph{embedding problem}. The problem has found many applications across various fields since it was first introduced by Elfving \cite{elfving}. Those include evolutionary biology \cite{jiabio},\cite{bio}, economics \cite{israelgenerator}, and social science \cite{singer}. Motivated by applications, algorithms for computing matrix roots have been developed \cite{MR2134331,MR2262981,MR4296928}, with some specifically tailored for stochastic matrices  \cite{MR4734557}. 

In this work, we introduce the class of \emph{arbitrarily finely divisible} stochastic matrices: stochastic matrices that have a stochastic $c$-th root for infinitely many $c \in \N$. While this is a weaker condition than infinite divisibility, it still implies the existence of transition matrices for arbitrarily short time intervals in the context of Markov chains.

Our work builds on extensive literature on stochastic $c$-th roots of stochastic matrices, and the embedding problem  in particular.  An excellent overview of the literature on stochastic roots of stochastic matrices and their connection to Markov chains is provided in \cite{highamlin}. Specifically, roots of $2\times 2$ and $3\times 3$ stochastic matrices are investigated in \cite{hegunn}. The $3\times 3$ case is also studied in \cite{guerry1} and \cite{guerry2}. In \cite{highamlin} it is shown that the stochastic matrices that are inverses of $M$-matrices and a specific family of symmetric positive semi-definite stochastic matrices have a $c$-th stochastic root for every $c\in \N$. Following \cite{highamlin}, some conditions under which positive semi-definite doubly stochastic matrices have a $c$-th root for every $c\in \N$ were given in \cite{nader}.

For an arbitrarily finely divisible matrix $A$, the set of its stochastic roots necessarily contains an accumulation point that shares some spectral properties with $A$. In order to understand this connection, we study the problem over the complex numbers in Section \ref{sec:complexlimit}, after we establish the notation used in this work in Section \ref{sec:background}. In Section \ref{sec:L_+(A)} we focus our attention on stochastic matrices, where the set of possible accumulation points, at least in the nonsingular case, can be effectively described. The theory developed in Sections  \ref{sec:complexlimit} and \ref{sec:L_+(A)} is applied to specific cases in the following sections. We revisit the case of $2 \times 2$ matrices in Section \ref{sec:2x2}, where the small dimension allows us to explicitly illustrate all the concepts developed earlier. If a nonsingular stochastic matrix has an irreducible accumulation point, then it has to be a circulant matrix. This observation serves as our motivation to study $3 \times 3$ circulant matrices in Section \ref{sec:3x3}. Finally, we characterise the set of arbitrarily finely divisible rank two matrices in Section \ref{sec:rank2}.

\section{Notation}\label{sec:background}

 We denote the set of natural numbers by $\N$ and for $a,b \in\N$ define $a\N+b=\{a k+b:k\in \N\}.$ The set $\N_0$ denotes the natural numbers including zero, i.e. $\N_0=\N \cup \{0\}.$ $\R$ denotes the set of all real numbers, and  $\C$ the set of all complex numbers. The modulus of $\lambda \in \C$ is denoted by $|\lambda|$, and for $c \in \N$ the set of $c$ distinct complex roots of $\lambda$ is denoted by $\{\lambda\}^{\frac{1}{c}}$. For $\lambda >0$, $\lambda^{\frac{1}{c}}$ denotes the positive element of $\{\lambda\}^{\frac{1}{c}}$. The cardinality of a set $\mc S$ is denoted by $|\mc S|$. For $n \in \N$, we define the set $$\mc {CP}(n):=\{x \in \N; \gcd(x,n)=1\}.$$

 We consider $\R^n$ and $\R_+^n$ to be the set of $n$-dimensional real vectors and nonnegative vectors, respectively.  $\mc M_n(\C)$ and $\mc M_n(\R)$ denote the set of $n\times n$ matrices with complex and real entries, respectively.
 The notation $\mc S_n^+$ represents the set of $n\times n$ stochastic matrices (nonnegative matrices with each row sum equal to $1$). The notation $\diag \npmatrix{d_1 & \ldots & d_n}$ refers to $n\times n$ diagonal matrix with diagonal entries $d_1,\ldots, d_n$.   $I_n$ denotes the $n\times n$ identity matrix, $0_{n}$ denotes the $n\times n$ zero matrix and $C_n$ denotes the adjacency matrix of a directed cycle of length $n$, where $n\in \N$ and vertices are listed in a natural cycle order. The vector $\mathbf{1}_n$ refers to the $n\times 1$ vector with all entries equal to $1.$ 

 The notation $A\simP B$ indicates that the matrix $A$ is permutationally similar to the matrix $B.$ For a matrix $M$,  $\sspan(M)$ denotes the vector space spanned by the columns of the matrix $M.$

The spectrum of a matrix $A$, which is the multiset of its eigenvalues listed with repetitions, is denoted by $\sigma(A)$. We use the notation $\sigma_*(A)$ to denote the set of distinct nonzero eigenvalues of $A$ (in this set the eigenvalues are not listed with repetitions). In addition, we use $\mult(\lambda, A)$ to denote the algebraic multiplicity of  $\lambda$ as the eigenvalue of $A$. 
The generalised eigenspace of $A$ corresponding to the eigenvalue $\lambda$ is represented by $\mc E(\lambda, A)$. The Jordan Canonical Form (JCF) of $A\in \mc M_n(\C)$ is denoted by $J(A)=J_*(A)\oplus N(A),$ where $N(A)$ collects together all the Jordan blocks corresponding to eigenvalue $0$, $J_*(A)=\oplus_{\lambda \in \sigma_*(A)}J(\lambda,A)$ and $J(\lambda,A)$ contains all the Jordan blocks for $\lambda\in \sigma_*(A)$. The multiset of block sizes in the JCF of $A$ corresponding to the eigenvalue $\lambda$ is denoted by $\mc B_{JCF}(\lambda,A)$, i.e. ${\mc B}_{JCF}(\lambda,A)=\{b_1,\ldots,b_k\}$, where $b_1,\ldots,b_k$ are block sizes in $J(\lambda,A)$.

A partition of a multiset is defined as a division of the multiset into non-overlapping sub-multisets whose union is the original multiset. For a multiset $\mc B$ of numbers, $\Sigma(\mc B)$ represents the sum of all elements in $\mc B.$

Having established predominantly standard notation, we proceed to define concepts that are specific to this work. For $\mc S \subseteq \mc M_n(\C)$ and $A \in \mc S$ we define: 
\begin{align*}
    \mc P_{\mc S}(A)&:=\{c\in \N: B^c=A\text{ for some } B\in \mc S \}\\
    \mc R_{\mc S}(A)&:=\{B: B^c=A, B\in \mc S, c \in \N\}.
\end{align*}
We use the following simplified notation when $\mc S$ is the set of complex square matrices, real square matrices, or stochastic matrices:
 \begin{align*}
  \mc P_{\C}(A)&:=\mc P_{\mc M_n(\C)}(A), & \mc P_{\R}(A)&:=\mc P_{\mc M_n(\R)}(A),  & \mc P_+(A)&:=\mc P_{\mc S_n^+}(A),\\
  \mc R_{\C}(A)&:=\mc R_{\mc M_n(\C)}(A), &\mc R_{\R}(A)&:=\mc R_{\mc M_n(\R)}(A),  &\mc R_+(A)&:=\mc R_{\mc S_n^+}(A)
 \end{align*}

 In our notation, a stochastic matrix is infinitely divisible, if  $\mc P_+(A)=\N$. In this work, we investigate a more general class of matrices with $|\mc P_+(A)|=\infty$.

\begin{definition}\label{def:limitL}
 Let $\mc S\subseteq \mc M_n(\C)$ and $A \in \mc S$. We say that $L$ is \emph{a limit of $\mc S$-roots of $A$} if there exists a strictly increasing sequence $(c_i)_{i=1}^{\infty} \subseteq \mc P_{\mc S}(A)$, and a convergent sequence $(B_i)_{i=1}^{\infty} \subseteq \mc R_{\mc S}(A)$ that satisfies $B_i^{c_i}=A$ and $\lim_{i \rightarrow \infty} B_i=L$. We define $\mc L_{\mc S}(A)$ to be the set of all limits of $\mc S$-roots of $A$. 
\end{definition}

As above, we use notation: 
\begin{align*}
\mc L_{\C}(A):=\mc L_{\mc M_n(\C)}(A), \,\, \mc L_{\R}(A)=\mc L_{\mc M_n(\R)}(A), \,\, \mc L_{+}(A)=\mc L_{\mc S_n^+}(A).
\end{align*}

\section{Properties of $L \in \mc L_{\C}(A)$}\label{sec:complexlimit}

In this section we aim to highlight the connections between $A \in \mc M_n(\C)$ and $L \in \mc L_{\C}(A)$.

The set $\mc R_{\C}(A)$ is well understood. In particular, $p$-th roots of matrices can be understood through their Jordan canonical form. The following result from \cite{highamlin} details the structure of matrix $p$-th roots for nonsingular matrices.

\begin{theorem}\cite{highamlin}\label{thm:higham_lin}
Let $A \in \mc M_n(\C)$ be a nonsingular matrix with Jordan canonical form $Z^{-1}AZ = J = \diag(J_1, J_2, \ldots, J_m)$, where each Jordan block $J_k = J_k(\lambda_k) \in \mc M_{m_k}(\C)$ corresponds to the eigenvalue $\lambda_k$. Suppose $s \leq m$ denotes the number of distinct eigenvalues of $A$.

For each $k = 1, 2, \ldots, m$, let $L_k^{(j_k)}(\lambda_k)$ denote the $p$-th roots of $J_k(\lambda_k)$, given by
\begin{equation}\label{eq:block_root}
L_k^{(j_k)}(\lambda_k) = \begin{pmatrix}
f_{j_k}(\lambda_k) & f_{j_k}'(\lambda_k) & \cdots & \frac{f_{j_k}^{(m_k-1)}(\lambda_k)}{(m_k-1)!} \\
& f_{j_k}(\lambda_k) & & \vdots \\
& & \ddots & f_{j_k}'(\lambda_k) \\
& & & f_{j_k}(\lambda_k)
\end{pmatrix},
\end{equation}
where $j_k \in \{1, 2, \ldots, p\}$ denotes the branch of the $p$-th root function $f(z) = \sqrt[p]{z}$.

Then, $A$ has exactly $p^s$ $p$-th roots that can be expressed as polynomials in $A$, given by
$$
X_j = Z \diag(L_1^{j_1}, L_2^{j_2}, \ldots, L_m^{j_m}) Z^{-1}, \quad j = 1, 2, \ldots, p^s,
$$
where $j_1, j_2, \ldots, j_m$ are chosen such that $j_i = j_k$ whenever $\lambda_i = \lambda_k$.

If $s < m$, then $A$ has additional $p$-th roots that form parameterized families
$$
X_j(U) = Z U \diag(L_1^{j_1}, L_2^{j_2}, \ldots, L_m^{j_m}) U^{-1} Z^{-1}, \quad j = p^s + 1, \ldots, p^m,
$$
where $j_k \in \{1, 2, \ldots, p\}$, $U$ is any nonsingular matrix commuting with $J$, and for each $j$, there exist indices $i$ and $k$ such that $\lambda_i = \lambda_k$ but $j_i \neq j_k$.
\end{theorem}

Also for singular matrices, the $p$-th roots are found by resorting to the Jordan canonical form. In this case, the singular part and the nonsingular part are considered separately.  

\begin{theorem}\cite{highamlin}\label{thm:complex_singular_roots}
    Let $A,B\in \mc M_n(\C)$ and $A=B^p$, $p\in \N$. Let $A$ have the Jordan canonical form $Z^{-1}AZ=J_*(A)\oplus N(A)$, where $N(A)$ collects together all the Jordan blocks corresponding to the eigenvalue $0$ and $J_*(A)$ contains the remaining Jordan blocks.     Then $B=Z\diag(X_0,X_1)Z^{-1},$ where $X_1$ is any $p$-th root of $J_*(A)$, characterized by Theorem \ref{thm:higham_lin}, and $X_0$ is any $p$-th root of $N(A).$
    \end{theorem}

The main focus of our work is to study irreducible stochastic matrices $A$ with $|\mc P_{+}(A)|=\infty$, but first, we introduce some concepts in a more general setting.

\begin{proposition}\label{prop:nilpotentpart0}
Let $A \in \mc M_n(\C)$, $|\mc P_{\C}(A)|=\infty$ and $L \in \mc L_{\C}(A)$. Then $N(A)=0_{n_0}$, where $n_0=\mult(0,A)$, and if $\lambda$ is an eigenvalue of $L$, then either $|\lambda|=1$ or $\lambda=0$. 
\end{proposition}

\begin{proof}
If $A, B\in \mc M_n(\C)$ so that $A=B^c$, then $N(A)$ is similar to $N(B)^c$ by Theorem \ref{thm:complex_singular_roots}. Let $n_0=\mult(0,A)=\mult(0,B)$. Since $N(B)^c=0_{n_0}$ for all $c \geq n_0$, the conclusion follows.

Next, let $(c_j)_{j=1}^{\infty}$ be a strictly increasing sequence of positive integers and $(B_j)_{j=1}^{\infty}$ a convergent sequence satisfying $B_j^{c_j}=A$ and $\lim_{j \rightarrow \infty} B_j=L$.  
 Let  $\lambda_k$, $k=1,\ldots,n$, be the eigenvalues of $A$ (listed with multiplicities). The absolute values of eigenvalues of $B_j$ are equal to $|\lambda_j|^{1/c_j}$, $k=1,\ldots,n$. Since the eigenvalues are continuous functions of the entries of the matrix, this implies that the eigenvalues $\nu_k$ of $L$ satisfy $|\nu_k|=\lim_{c_i\rightarrow \infty} |\lambda_k|^{\frac{1}{c_i}}$, $k=1,\ldots,n$. From this observation we conclude that $\nu_k$ is either equal to $0$ (when $\lambda_k=0$), or has $|\nu_k|=1$ (when $\lambda_k\neq 0$.)
 
\end{proof}

Given $A\in \mc M_n(\C)$ with $|\mc P_{\C}(A)|=\infty$, we want to understand the relationship between the (generalised) eigenspaces of $A$ and $L\in \mc L_{\C}(A)$. This relationship can be tracked through the multiplicity rearrangement matrix $M$ defined and studied below. 

\begin{definition}\label{def:multiplicity_rearrangement}
Let $A \in \mc M_n(\C)$ with $|\sigma_*(A)|=q$ and $L \in \mc L_{\C}(A)$ with $|\sigma_*(L)|=s$. \emph{The multiplicity rearrangement matrix for $(A,L)$} is a $q\times s$ matrix $M$ whose rows and columns are respectively indexed by the elements of $\sigma_*(A)$ and $\sigma_*(L)$ (in some fixed order), and is defined by $$M[\lambda,\nu]:=\dim(\mc E(\lambda,A) \cap \mc E(\nu, L)).$$ ($M[\lambda,\nu]$ denotes the element of $M$ in the row that corresponds to $\lambda$ ($\lambda$-row) and the column that corresponds to $\nu$ ($\nu$-column).)
\end{definition}

In the definition, the multiplicity rearrangement matrix for $(A,L)$ depends on the order of elements in $\sigma_*(A)$ and $\sigma_*(L)$. To remove ambiguity and fix the order, we will first list real elements in decreasing order then complex conjugate pairs ordered first by real parts, then by the absolute value of the imaginary parts. In examples, we will often add eigenvalue labeling of rows and columns to the matrix $M$ (separated by a line), for clarity. 

Our next results detail the relationships between the generalised eigenspaces of $A$ and $L \in \mc L_{\C}(A)$.  

\begin{theorem}\label{prop:M(A,L)complex}
Let $A \in \mc M_n(\C)$, $L \in \mc L_{\C}(A)$, and $M$ be the multiplicity rearrangement matrix for $(A,L)$.
\begin{enumerate}
    \item For each eigenvalue $\lambda$ of $A$ there exists a partition of the multi-set $\mc B_{JCF}(\lambda,A)$ into multisets $\mc B(\lambda,\nu)$,  $\nu \in \sigma_*(L)$, so that $M[\lambda,\nu]$ is equal to the sum of elements in $\mc B(\lambda,\nu)$. (If $\mc B(\lambda,\nu)$ is the empty set, then $M[\lambda,\nu]:=0$.) In particular, the sum of elements in the $\lambda$-row of $M$ is equal to $\mult(\lambda,A)$. 
    \item The sum of elements in $\nu$-column of $M$ is equal to $\mult(\nu,L)$. 
 \item For $\lambda \in \sigma_*(A)$ we have 
    $$\mc E(\lambda, A) \subseteq \sspan\{\mc E(\nu, L); M[\lambda,\nu]\neq 0\},$$
     and 
      if all $\nu \in \sigma_*(L)$ with $M[\lambda,\nu] \neq 0$ satisfy $\mult(\nu,L)=M[\lambda,\nu]$,
     then the two sets are the same. 
    \item For $\nu \in \sigma_*(L)$ we have 
    $$\mc E(\nu, L) \subseteq \sspan\{\mc E(\lambda, A); M[\lambda,\nu]\neq 0\},$$ 
    and if all $\lambda \in \sigma_*(A)$ with $M[\lambda,\nu] \neq 0$ satisfy $\mult(\lambda,A)=M[\lambda,\nu]$,
     then the two sets are the same. 
    \end{enumerate}
\end{theorem}

\begin{proof}

Let $B$ be a matrix that satisfies $B^c=A$. From Theorem \ref{thm:higham_lin} it follows that for each $\lambda \in \sigma_*(A)$ the multiset $\mc B_{JCF}(\lambda, A)$ partitions into multisets $\mc B_{JCF}(\mu, B)$, $\mu \in \{\lambda\}^{\frac{1}{c}}$.
(Since $|\{\lambda\}^{\frac{1}{c}}|=c$ and some of the multisets $\mc B_{JCF}(\mu,B)$ can be empty, this partition has at most $c$ parts.) In particular,  $\sigma_*(B)\subseteq \cup_{\lambda \in \sigma_*(A)}\{\lambda\}^{\frac{1}{c}}$. Furthermore, $$\mc E(\lambda,A)=\sspan\{\mc E(\mu, B); \mu \in \{\lambda\}^{\frac{1}{c}}\}.$$ 

Let $(B_i)_{i=1}^{\infty}$ be a convergent sequence of matrices that satisfies $B_i^{c_i}=A$ and $\lim_{i\rightarrow \infty}B_i=L$. As noted above, for each $\lambda \in \sigma_*(A)$ and $i$, the multiset $\mc B_{JCF}(\lambda, A)$ partitions into multisets $\mc B_{JCF}(\mu,B_i)$, $\mu \in \{\lambda\}^{\frac{1}{c_i}}$. Also, since the eigenvalues are continuous functions of matrix entries, for each $i$, we can partition $\sigma_*(B_i)$ into $s$ multisets $\sigma(i,\nu)$ so that elements of $\sigma(i,\nu)$ converge to $\nu$ as $i$ goes to infinity. (We can define $\sigma(i,\nu)$ more formally as follows. Choose disjoint open neighbourhoods $\mc U_{\nu}$ of $\nu \in \sigma_*(L)$, and define $\sigma(i,\nu):=\sigma_*(B_i) \cap \mc U_{\nu}$. By considering a sub-sequence of $B_i$ we can assume that $\sigma(i,\nu)$ defined in this way partition $\sigma_*(B_i)$ for all $i \in \N$.)

Let $Z^{-1}AZ=(\oplus_{\lambda \in \sigma_*(A)} J(\lambda,A))\oplus 0_{n_0}$ ( Proposition \ref{prop:nilpotentpart0}). For each $\lambda \in \sigma_*(A)$, let $Z_{\lambda}$ be the submatrix of $Z$ consisting of the columns of $Z$ that belong to $\mc E(\lambda,A)$ so that $AZ_{\lambda}=Z_{\lambda}J(\lambda,A)$. By Theorem \ref{thm:higham_lin} we have: 
\begin{equation*}
B_iZ_{\lambda}U_{\lambda}(i)=
Z_{\lambda}U_{\lambda}(i)\left(\oplus_{\mu \in \{\lambda\}^{\frac{1}{c_i}}} K_{c_i}(\mu, B_i)\right),
\end{equation*}
where $K_{c_i}(\mu, B_i)$ is a direct sum of blocks of the form \eqref{eq:block_root} for the function $f(z)=z^{\frac{1}{c_i}}$ and the eigenvalue $\mu$, and $U_{\lambda}(i)$ are invertible matrices that commute with $J(\lambda,A)$. 
In particular, $\left(\oplus_{\mu \in \{\lambda\}^{\frac{1}{c_i}}} K_{c_i}(\mu, B_i)\right)^{c_i}=J(\lambda,A)$. 
For $\nu \in \sigma_*(L)$ we define: 
$$K_{c_i}(\lambda,\nu)=\oplus_{\mu \in \{\lambda\}^{\frac{1}{c_i}}\cap \sigma(i,\nu)} K_{c_i}(\mu, B_i),$$
and by possibly moving to a sub-sequence, we assume that the sizes of blocks (in some fixed order) in $K_{c_i}(\lambda,\nu)$ are the same for all $i \in \N$. Let us denote the multiset of those block sizes as $\mc B(\lambda,\nu)$, and let $\Sigma(\mc B(\lambda,\nu))$ denote the sum of all elements in $\mc B(\lambda,\nu)$. From
$$B_iZ_{\lambda}=Z_{\lambda}U_{\lambda}(i)(\oplus_{\nu \in \sigma_*(L)} K_{c_i}(\lambda, \nu))U_{\lambda}(i)^{-1},$$
the fact that $\{B_i\}_{i=1}^{\infty}$ is a convergent sequence and $Z_{\lambda}$ has full rank, we deduce that $\{U_{\lambda}(i)(\oplus_{\nu \in \sigma_*(L)} K_{c_i}(\lambda, \nu))U_{\lambda}(i)^{-1}\}_{i=1}^{\infty}$ is a convergent sequence, and
we define: 
$$\hat K(\lambda):=\lim_{i \rightarrow \infty}\left(U_{\lambda}(i)(\oplus_{\nu \in \sigma_*(L)} K_{c_i}(\lambda, \nu))U_{\lambda}(i)^{-1}\right).$$
Observe:
$$LZ_{\lambda}=Z_{\lambda}\hat K(\lambda),$$
and $$Z^{-1}LZ=\left(\oplus_{\lambda \in \sigma_*(A)} \hat K(\lambda)\right)\oplus N',$$
where $N'$ is a nilpotent matrix. 
 In particular,  $\sigma_*(\hat K(\lambda))\subset \sigma_*(L)$ and $\mult(\nu,\hat K(\lambda))=\Sigma(\mc B(\lambda,\nu))$. Writing $\hat K(\lambda)$ in its Jordan canonical form with similarity transformation matrix $T_{\lambda}$:
 $$T_{\lambda}^{-1}\hat K(\lambda)T_{\lambda}=\oplus_{\nu \in \sigma_*(L)}J(\nu,\hat K(\lambda))$$
 and noting 
 \begin{align*}
 \mc E(\nu,L) \cap \sspan(Z_{\lambda}T_{\lambda}^{-1})=\mc E(\nu,L) \cap \sspan(Z_{\lambda})=\mc E(\nu,L) \cap \mc E(\lambda,A),
 \end{align*}
we conclude that $$M[\lambda,\nu]:=\dim(\mc E(\nu, L) \cap \mc E(\lambda,A))=\Sigma(\mc B(\lambda,\nu)).$$
Furthermore: 
$$\mult(\lambda,A)=\dim(\mc E(\lambda,A))=\sum_{\nu \in \sigma_*(L)}\dim(\mc E(\nu, L) \cap \mc E(\lambda,A))=\sum_{\nu \in \sigma_*(L)}M[\lambda,\nu]$$ 
and
\begin{align*}
\mc E(\lambda,A)&=\sspan\{\mc E(\lambda, A) \cap \mc E(\nu,L);  \nu \in \sigma_*(L)\}\\
&=\sspan\{\mc E(\nu, L) \cap \mc E(\lambda,A); \nu \in \sigma_*(L) \text{ and } M[\lambda,\nu]\neq 0\} \\
&\subseteq \sspan\{\mc E(\nu,L); \nu \in \sigma_*(L) \text{ and } M[\lambda,\nu]\neq 0\}. 
\end{align*}
Note that the last inclusion is an equality, if
all $\nu \in \sigma_*(L)$  with $M[\lambda,\nu] \neq 0$ satisfy $\mult(\nu,L)=M[\lambda,\nu]$.

Similarly: 
$$\mult(\nu,L)=\dim(\mc E(\nu,L))=\sum_{\lambda \in \sigma_*(A)}\dim(\mc E(\nu, L) \cap \mc E(\lambda,A))=\sum_{\lambda \in \sigma_*(A)}M[\lambda,\nu]$$ 
and
\begin{align*}
\mc E(\nu,L)&=\sspan\{\mc E(\nu, L) \cap \mc E(\lambda,A);  \lambda \in \sigma_*(A)\}\\
&=\sspan\{\mc E(\nu, L) \cap \mc E(\lambda,A); \lambda \in \sigma_*(A) \text{ and } M[\lambda,\nu]\neq 0\} \\
&\subseteq \sspan\{\mc E(\lambda,A); \lambda \in \sigma_*(A) \text{ and } M[\lambda,\nu]\neq 0\}. 
\end{align*}
As above, we observe that the last inclusion is an equality, if
all $\lambda \in \sigma_*(A)$  with $M[\lambda,\nu] \neq 0$ satisfy $\mult(\lambda,A)=M[\lambda,\nu]$.
\end{proof}

\begin{corollary}\label{cor:A,L simple eigs}
Let $A\in \mc M_n(\C)$ be a nonsingular matrix, $L \in \mc L_{\C}(A)$, and $M$ the multiplicity rearrangement matrix for $(A,L)$.
\begin{enumerate}
\item If $L$ has only simple eigenvalues then $\mc E(\lambda, A) = \sspan\{\mc E(\nu, L); M[\lambda,\nu]\neq 0\}$ for every $\lambda \in \sigma_*(A)$. 
\item If $A$ has only simple eigenvalues then $\mc E(\nu, L) = \sspan\{\mc E(\lambda, A); M[\lambda,\nu]\neq 0\}$ for every $\nu \in \sigma_*(L)$.
\end{enumerate}
If either of the statements above hold, then there exists an invertible matrix $Z$ so that $Z^{-1}AZ$ and $Z^{-1}LZ$ are diagonal matrices, i.e. $A$ and $L$ are simultaneously diagonalisable. 
\end{corollary}
\begin{proof}
Item 1. follows directly from item 3. in Theorem \ref{prop:M(A,L)complex} by observing that if $\nu \in \sigma_*(L)$ is a simple eigenvalue of $L$, then $\mult(\nu,L)=M[\nu,\lambda]$ if and only if $M[\nu,\lambda] \neq 0$. Item 2. can be deduced from item 4. in Theorem \ref{prop:M(A,L)complex} in a similar way. 

In both cases, $M$ is a $0$-$1$-matrix, hence $A$ and $L$ have to be diagonalisable matrices since all the sets $\mc B(\lambda,\nu)$ in item 1 of Theorem \ref{prop:M(A,L)complex} have to either be empty or equal to $\{1\}$.
Under the assumptions of item 1. we have $Z^{-1}LZ$ diagonal implies $Z^{-1}AZ$ diagonal. Similarly, under the assumptions of item 2., $Z^{-1}AZ$ diagonal implies $Z^{-1}LZ$ diagonal. 
\end{proof}

\begin{example}
Let 
{\footnotesize $$A=\lambda \left(
\begin{array}{cccccc}
 1 & 0 & 0 & 0 & 0 & 0 \\
 0 & 1 & 0 & 0 & 0 & 0 \\
 0 & 0 & 1 & 0 & 0 & 0 \\
 0 & 0 & 0 & -1 & 0 & 0 \\
 0 & 0 & 0 & 0 & -1 & 0 \\
 0 & 0 & 0 & 0 & 0 & -1 \\
\end{array}
\right), \, B_j'=\lambda^{\frac{1}{2j}}\left(
\begin{array}{cccccc}
 1 & 0 & 0 & 0 & 0 & 0 \\
 0 & 1 & 0 & 0 & 0 & 0 \\
 0 & 0 & -1 & 0 & 0 & 0 \\
 0 & 0 & 0 & -e^{\frac{i \pi }{2 j}} & 0 & 0 \\
 0 & 0 & 0 & 0 & -e^{\frac{i \pi }{2 j}} & 0 \\
 0 & 0 & 0 & 0 & 0 & e^{\frac{i \pi }{2 j}} \\
\end{array}
\right)$$}
where $\lambda>0$, and $j \in \N$. We have
\begin{align*}
\sigma_*(A)=\{\lambda, -\lambda\},\, \sigma_*(B_j')=\{\lambda^{\frac{1}{2j}},-\lambda^{\frac{1}{2j}},-\lambda^{\frac{1}{2j}}e^{\frac{i \pi }{2 j}},\lambda^{\frac{1}{2j}}e^{\frac{i \pi }{2 j}}\}.
\end{align*}
Since $(B_j')^{2j}=A$ we have $B_j' \in \mc R_{\C}(A)$. The sequence $(B_j')_{j=1}^{\infty}$ converges to the matrix:
{\footnotesize $$L'=\left(
\begin{array}{cccccc}
 1 & 0 & 0 & 0 & 0 & 0 \\
 0 & 1 & 0 & 0 & 0 & 0 \\
 0 & 0 & -1 & 0 & 0 & 0 \\
 0 & 0 & 0 & -1 & 0 & 0 \\
 0 & 0 & 0 & 0 & -1 & 0 \\
 0 & 0 & 0 & 0 & 0 & 1 \\
\end{array}
\right)$$}
with $\sigma_*(L')=\{1,-1\}$. The multiplicity rearrangement matrix for $(A,L')$ is equal to 
$$M'=\left(
\begin{array}{c|cc}
  & 1 &  -1 \\
  \hline
 \lambda  & 2 & 1 \\
 -\lambda  & 1 & 2 \\
\end{array}
\right).$$
Since all the matrices involved are diagonal, the associated eigenspaces are straightforward to compute.

Note that for any $3 \times 3$ matrices $S$ and $S'$ the matrix $Z:=S\oplus S'$ commutes with $A$. Let  
$$Z_j:=\left(
\begin{array}{ccc}
 1 & \frac{1}{j}+1 & -\frac{1}{j^2} \\
 0 & 1 & \frac{j-1}{j} \\
 2 & 0 & 1 \\
\end{array}
\right) \oplus \left(
\begin{array}{ccc}
 1 & \frac{3}{j}+2 & 5-\frac{9}{j^2} \\
 0 & 1 & 2-\frac{3}{j} \\
 \frac{2}{3} & 0 & 1 \\
\end{array}
\right)$$
and $B_j=Z_jB_j'Z_j^{-1}$.  
Then $(B_j)^{2j}=A$, $B_j \in \mc R_{\C}(A)$ and the sequence $(B_j)_{j=1}^{\infty}$ converges to the matrix
{\footnotesize $$L=\left(
\begin{array}{cccccc}
 1 & 0 & 0 & 0 & 0 & 0 \\
 \frac{4}{3} & -\frac{1}{3} & -\frac{2}{3} & 0 & 0 & 0 \\
 \frac{4}{3} & -\frac{4}{3} & \frac{1}{3} & 0 & 0 & 0 \\
 0 & 0 & 0 & -21 & 40 & 30 \\
 0 & 0 & 0 & -8 & 15 & 12 \\
 0 & 0 & 0 & -4 & 8 & 5 \\
\end{array}
\right)$$}
with $\sigma_*(L)=\{1,-1\}$ and 
\begin{align*}
\mc E(1,L)=\sspan{\footnotesize\left(
\begin{array}{ccc}
 0 & \frac{1}{2} & 1 \\
 0 & 0 & 1 \\
 0 & 1 & 0 \\
 5 & 0 & 0 \\
 2 & 0 & 0 \\
 1 & 0 & 0 \\
\end{array}
\right)}, \, 
\mc E(-1,L)=\sspan{\footnotesize\left(
\begin{array}{ccc}
 0 & 0 & 0 \\
 0 & 0 & 1 \\
 0 & 0 & 1 \\
 \frac{3}{2} & 2 & 0 \\
 0 & 1 & 0 \\
 1 & 0 & 0 \\
\end{array}
\right)}.
\end{align*}
The multiplicity rearrangement matrix $M$ for $(A,L)$ is equal to $M'$. 

Following the notation in the proof of Theorem \ref{prop:M(A,L)complex} we have $\sigma(j,1)=\{\lambda^{\frac{1}{2j}},\lambda^{\frac{1}{2j}}e^{\frac{i \pi }{2 j}}\}$, $\sigma(j,-1)=\{-\lambda^{\frac{1}{2j}},-\lambda^{\frac{1}{2j}}e^{\frac{i \pi }{2 j}}\}$ and
\begin{align*}
\mc E(\sigma(j,1),B_j)&=\sspan {\footnotesize \left(
\begin{array}{ccc}
 \frac{1}{2} & \frac{1}{j}+1 & 0 \\
 0 & 1 & 0 \\
 1 & 0 & 0 \\
 0 & 0 & 5-\frac{9}{j^2} \\
 0 & 0 & 2-\frac{3}{j} \\
 0 & 0 & 1 \\
\end{array}
\right)}\\
\mc E(\sigma(j,-1),B_j)&=\sspan\footnotesize{\left(
\begin{array}{ccc}
 -\frac{1}{j^2} & 0 & 0 \\
 \frac{j-1}{j} & 0 & 0 \\
 1 & 0 & 0 \\
 0 & \frac{3}{2} & \frac{3}{j}+2 \\
 0 & 0 & 1 \\
 0 & 1 & 0 \\
\end{array}
\right)}.
\end{align*}
Clearly, $\mc E(\sigma(j,1),B_j)$ converges to $\mc E(1,L)$ and  $\mc E(\sigma(j,-1),B_j)$ converges to $\mc E(-1,L)$. This example illustrates, that the spaces $\mc E(\sigma(j,\nu),B_j)$, $\nu \in \sigma_*(L)$ may change with $j$, but the dimensions $\dim(\mc E(\sigma(j,\nu),B_j)\cap \mc E(\lambda,A))$, $\lambda \in \sigma_*(A)$, are constant. 
\end{example}

Next, we consider real matrices $A$, and their real roots. For $A \in \mc M_n(\R)$ the non-real eigenvalues come in complex conjugate pairs $(\lambda, \bar{\lambda})$, where $\mc B_{JCF}(\lambda,A)=\mc B_{JCF}(\bar{\lambda},A)$. This property needs to be preserved when taking roots of $A$. Next result details how this is reflected in the multiplicity rearrangement matrix.

\begin{proposition}\label{prop:M(A,L)real}
Let $A \in \mc M_n(\R)$, $L \in \mc L_{\R}(A)$, and let $M$ be the multiplicity rearrangement matrix for $(A,L)$. Let $(B_i)_{i=1}^{\infty}$ be the sequence satisfying $B_i^{c_i}=A$ and $\lim_{i\rightarrow \infty}B_i=~L.$ 
\begin{enumerate}
\item Let $\lambda,\bar{\lambda} \in \sigma_*(A)$ be a complex conjugate pair of non-real eigenvalues of $A$ and $\nu \in \sigma_*(L)$.
\begin{itemize}
\item If $\nu \in \{1,-1\}$, then  $M[\lambda,\nu]=M[\bar{\lambda},\nu]$. 
\item If $\nu \not\in \{1,-1\}$, then  $M[\lambda,\nu]=M[\bar{\lambda},\bar{\nu}]$.
\end{itemize}
\item Let $\lambda \in \sigma_*(A) \cap \R$ and $\nu \in \sigma_*(L) \cap (\C\setminus\R)$.  Then $M[\lambda,\nu]=M[\lambda,\bar{\nu}]$. 
\item At least one of the statements below has to hold for all positive eigenvalues $\lambda$ and negative eigenvalues $-\lambda'$ of $A$. \begin{itemize}
\item $|(c_i)_{i=1}^{\infty}\cap 2 \N|=\infty$ and $M[-\lambda',1]$, $M[-\lambda',-1]$ are both even. 
\item $|(c_i)_{i=1}^{\infty}\cap 2 \N+1|=\infty$ and $M[\lambda,-1]$, $M[-\lambda',1]$ are both even. 
\end{itemize}
\end{enumerate}
\end{proposition}

\begin{proof}
We use the notation and notions developed in the proof of Theorem \ref{prop:M(A,L)complex}. 
\begin{enumerate}
    \item Let $\lambda \in \sigma_*(A)$ be a non-real eigenvalue of $A$. Then  $\bar{\lambda} \in \sigma_*(A)$, and for all $\nu \in \sigma_*(L)$, we have $\mc B(\lambda,\nu)=\mc B(\bar{\lambda},\bar{\nu})$ and $\mc B(\bar{\lambda},\nu)=\mc B(\lambda,\bar{\nu})$. 
    For $\nu \notin \{1,-1\}$ this implies $\bar \nu \in \sigma_*(L)$, $M[\lambda,\nu]=M[\bar \lambda,\bar \nu]$ and $M[\bar \lambda,\nu]=M[\lambda,\bar \nu]$. 
    For $\nu \in \{1,-1\}$ we have $\nu=\bar \nu$ and $M[\lambda,\nu]=M[\bar \lambda,\nu]$. 
    \item Let $\lambda \in \sigma_*(A) \cap \R$ and $\nu \in \sigma_*(L) \cap (\C\setminus\R)$. Then $\bar \nu \in \sigma_*(L)$ and $\mc B(\lambda,\nu)=\mc B(\lambda,\bar \nu)$, hence $M[\lambda,\nu]=M[\lambda,\bar \nu]$. 
    \item Note that at least one of the conditions $|(c_i)_{i=1}^{\infty}\cap 2 \N|=\infty$ and $|(c_i)_{i=1}^{\infty}\cap 2 \N+1|=\infty$ has to hold, and assume first that $|(c_i)_{i=1}^{\infty}\cap 2 \N|=\infty$. By possibly moving to a sub-sequence we can assume  $(c_i)_{i=1}^{\infty}\subseteq 2 \N$, and write $c_i=2 c_i'$. If $\mu \in \{-\lambda'\}^{\frac{1}{2 c_i'}}$ with $\mc B_{JCF}(\mu,B_i) \neq \emptyset$, then  $\bar{\mu} \in \{-\lambda'\}^{\frac{1}{2 c_i}}$ and 
    $\mc B_{JCF}(\mu,B_i)=\mc B_{JCF}(\bar{\mu},B_i)$. 
   (In particular, $\mult(-\lambda',A)$ has to be even.)
    Furthermore, if  for $\nu \in \{1,-1\}$ we have $\mu \in \sigma(i,\nu)$, then $\bar \mu \in \sigma(i,\nu)$. Consequently, each of the sets $\mc B(\lambda',1)$, $\mc B(\lambda', -1)$ has every element appearing even number of times, and the entries $ M[-\lambda',1] $ and $M[-\lambda',-1]$ are even.
    
   Next, let $|(c_i)_{i=1}^{\infty}\cap 2 \N+1|=\infty$, and as above let us assume $(c_i)_{i=1}^{\infty}\subseteq 2\N+1$ and $c_i=2c_i'+1.$ The set  $\{\lambda\}^{\frac{1}{2c_i'+1}} $ consist of a positive real number and complex conjugate pairs, while the set $\{-\lambda'\}^{\frac{1}{2c_i'+1}} $ consists of a negative real number and complex conjugate pairs. Hence, if $\mu \in \sigma(i,-1)\cap \{\lambda\}^{\frac{1}{2c_i'+1}}$, then $\bar{\mu} \in \sigma(i,-1)\cap \{\lambda\}^{\frac{1}{2c_i'+1}}$, and $\mc B(\lambda,-1)$ has to have every element appearing even number of times (blocks coming from $\mc B_{JCF}(\mu,B_i)$ and $\mc B_{JCF}(\bar{\mu},B_i)$ have to be of equal size). This implies $M[\lambda,-1]$ has to be even. Similarly, if $\mu \in \sigma(i,1)\cap \{-\lambda'\}^{\frac{1}{2c_i'+1}}$, then $\bar{\mu} \in \sigma(i,1)\cap \{-\lambda'\}^{\frac{1}{2c_i'+1}}$, and $\mc B(-\lambda',1)$ has to have every element appearing even number of times, implying $M[-\lambda',1]$ is even. 
\end{enumerate}
\end{proof}

\begin{proposition}\label{prop:(A,L)stoc}
    Let $A \in \mc \mc S_n^+$ be irreducible, $L \in \mc L_+(A)$, and let $M$ be the multiplicity rearrangement matrix for $(A,L)$. Then the first row and the first column of $M$ both correspond to the eigenvalue $1$, $M[1,1]=1$ and $M[1,\nu]=0$ for all $\nu \in \sigma_*(L)\setminus \{1\}$. 
\end{proposition}

\begin{proof}
Let $(B_i)_{i=1}^{\infty}$ be the sequence of irreducible stochastic matrices satisfying $B_i^{c_i}=A$ and $\lim_{i\rightarrow \infty}B_i=~L.$ Then (as defined in the proof of Theorem \ref{prop:M(A,L)complex}) we have $\sigma(i,1)\cap \{1\}^{\frac{1}{c_i}}=\{1\}$ for all $i$, $\mc B(1,1)=\{1\}$, and $\mc B(1,\nu)=\emptyset$ for all $\nu \in \sigma_*(L)\setminus \{1\}$. The statement follows.  
\end{proof}

\begin{remark}\label{rem:left eigenspaces}
Although results in this section only consider right generalised eigenspaces of matrices ($\mc E(\lambda,A)$, $\mc E(\nu, L)$), note that equivalent statements hold for left generalised eigenspaces ($\mc E_{\ell}(\lambda,A)$, $\mc E_{\ell}(\nu, L)$). 
\end{remark}

\section{$\mc P_+(A)$ and $\mc L_+(A)$}\label{sec:L_+(A)}

We now direct our research towards irreducible stochastic matrices $A$ satisfying $|\mc P_+(A)|=\infty$. Note, every irreducible nonnegative matrix is similar to an irreducible stochastic matrix via diagonal similarity. Hence, the study of nonnegative roots of nonnegative irreducible matrices can be reduced to the study of stochastic roots of stochastic irreducible matrices.

Let $\mc S\subset \mc M_n(\C)$ be closed under multiplication, $A \in \mc S$, and $a,b \in \N$. If $a \in \mc P_{\mc S}(A)$ and $b$ divides $a$, then it is easy to see that $b \in \mc P_{\mc S}(A)$. Although the set of irreducible stochastic matrices is not closed under multiplication, the same observation holds. 
Recall, the peripheral spectrum of $A$ (referred to in the proof below) is the multiset $\{\lambda\in \sigma(A):|\lambda|=\rho(A)\},$ where $\rho(A)$ denotes the spectral radius of $A$.

\begin{lemma}\label{lem:P+_basic}
Let $A \in \mc S_n^+$ be an irreducible matrix with index $h$.
\begin{enumerate}
\item If $a \in \mc P_{+}(A)$ and $b \in \N$ divides $a$, then $b \in \mc P_{+}(A)$. 
\item $\mc P_+(A) \subseteq \mc{CP}(h)$. 
\end{enumerate}
\end{lemma}

\begin{proof}
    To prove the first item let us write $a=mb$ for some $m \in \N$. If
    $A=B^a=(B^m)^b$ for $B \in \mc S_n^+$, then both $B$ and $B^m$ have to be irreducible stochastic matrices, and the conclusion follows. 

    Now let $c \in \mc P_+(A)$, and $B \in \mc S_n^+$ with $B^c=A$. Then the peripheral spectrum of both $A$ and $B$ is equal to $\{1\}^{\frac{1}{h}}$. On the other hand, the multiset $\{\zeta^c; \zeta \in \{1\}^{\frac{1}{h}}\}$ also has to be equal to the peripheral spectrum of $A$. Since $\{\zeta^c; \zeta \in \{1\}^{\frac{1}{h}}\}=\{1\}^{\frac{1}{h}}$ if and only if $\gcd(c,h)=1$, the statement follows. 
\end{proof}

\begin{lemma}\label{lem:all_eigs_abs_1}
    If $A \in \mc S_n^+$ is nonsingular and irreducible with the property that every eigenvalue $\lambda$ of $A$ satisfies $|\lambda|=1$, then $A\simP C_n$.
\end{lemma} 
\begin{proof}
The result follows directly from the  Frobenius normal form for nonnegative matrices, see for example~\cite[Theorem~2.20]{berman}.
   \end{proof}

From now on we will focus on irreducible stochastic matrices $A\in \mc S_n^+$ that satisfy $|\mc P_+(A)|=~\infty$. 

\begin{definition}
    A matrix $A \in \mc S_n^+$ is called \emph{arbitrarily finely divisible} if $|\mc P_+(A)|=~\infty$.
\end{definition}

 Recall, for an irreducible matrix $A\in \mc S_n^+$, the left Perron vector $w\trans$ such that $w\trans A=w\trans$ and $w\trans 
 \mathbf{1}_n=\mathbf{1}_n$ is called \emph{the stationary distribution vector} of $A.$

\begin{proposition}\label{prop:on limits}
    Let $A$ be an irreducible stochastic matrix with $|\mc P_+(A)|=\infty$ and the stationary distribution vector $w\trans$. Then  $\mc L_+(A) \neq \emptyset$, and any $L \in \mc L_+(A)$ has the following properties:
    \begin{enumerate}
    \item $L$ is a (not necessarily irreducible) stochastic matrix satisfying $w\trans L=w\trans$.
    \item $L\simP \oplus_{i=1}^t L_{ii}$, where $t$ is the multiplicity of eigenvalue $1$ in $\sigma(L)$, and for $i=1,\ldots,t$, $L_{ii}$ is an irreducible stochastic matrix.
    \end{enumerate}
\end{proposition}
\begin{proof}
By Theorem \ref{prop:M(A,L)complex}, Remark \ref{rem:left eigenspaces} and Proposition \ref{prop:(A,L)stoc}, we have $\mc E_{\ell}(1,A) \subseteq \mc E_{\ell}(1,L)$, proving the first item. 
Now that we know that $L$ has a positive right and a positive left eigenvector corresponding to the eigenvalue $1$, item 2. follows from ~\cite[Theorem~3.14]{berman}. 
\end{proof}

\begin{example}\label{ex:singularlimit}
      Let $A\in \mc S_4^+$ be irreducible and singular with two nonzero eigenvalues. Then for $L\in \mc L_+(A)$ we have $\sigma(L)=\{1,-1,0,0\}$ or $\sigma(L)=\{1,1,0,0\}$. In particular, $\sigma(L)=\{1,-1,0,0\}$ if and only if $L$ is irreducible. In this case $$L\simP \left( \begin{array}{cc}
      0 & L_1\\
      L_2 & 0\\
      \end{array}\right)$$  by~\cite[Theorem~2.20]{berman}.
    If the zero blocks on the diagonal are $1\times 1$ and $3\times 3$ then: $$L\simP K=\left(
      \begin{array}{cccc}
        0 & a & b & 1-a-b  \\
        1 & 0 & 0 & 0 \\
        1 & 0 & 0 & 0 \\
        1 & 0 & 0 & 0\\
      \end{array}\right).$$
Note that $K$ has nonzero eigenvalues $1$ and $-1$ for all choices for $a$ and $b$ satisfying $a,b, a+b \in (0,1)$.    
If both the diagonal blocks are $2\times 2$ , then:  $$L\simP K=\left(
      \begin{array}{cccc}
        0 & 0 & s & 1-s  \\
        0 & 0 & t & 1-t \\
        u & 1-u & 0 & 0 \\
        v & 1-v & 0 & 0\\
      \end{array}\right)$$
      for some $s,t,u,v \in (0,1)$, where
       $s=t$ or $u=v$ needs to hold for $\sigma(L)=~\{1,-1,0,0\}$. Without loss of generality, let $s=t$. Then,
     $$\text{rank}(L)=\begin{cases}
3 &\text{ if } u\neq v \\ 
2  &\text{ if } u= v\\
\end{cases}.$$
Note that if $\mathrm{rank}(L)=3$, then $L$ has a nontrivial nilpotent Jordan block. 

 If $\sigma(L)=\{1,1,0,0\}$, then $L\simP L_{1} \oplus L_{2}$, where $L_1$ and $L_2$ are irreducible stochastic matrices. If $L_{1}$ and $L_2$ are both $2\times 2$, then they both have to have rank $1$. Finally, we can have $L_1=(1)$ and $L_{22}$ an irreducible stochastic matrix with eigenvalues $\{1,0,0\}$. 
\end{example}

\begin{theorem}\label{thm:nonsinglimit}
    Let $A \in \mc S_n^+$ be nonsingular and irreducible with $|\mc P_+(A)|=\infty$, and $L\in\mc L_+(A)$. Then  $L \simP \bigoplus_{i=1}^{t}C_{n_i},$ where $\sum_{i=1}^t n_i=n$,  and $t$ is the multiplicity of $1$ in the spectrum of $L$.
\end{theorem}

\begin{proof}
The nonsingularity of $A\in \mc S_n^+$ ensures that $L\in \mc L_+(A)$ is nonsingular, hence all the eigenvalues $\lambda$ of $L$ satisfy $|\lambda|=1$ (Proposition \ref{prop:nilpotentpart0}). By item $2.$ of Proposition \ref{prop:on limits}  we have $L\simP \bigoplus_{i=1}^t L_{ii}$, where each $L_{ii}$ is an irreducible stochastic matrix of order $n_i$. Applying Lemma \ref{lem:all_eigs_abs_1} to each $L_{ii}$, we get $L_{ii}\simP C_{n_i}$. 
Thus,  $L\simP\bigoplus_{i=1}^t C_{n_i}$.
\end{proof}

Let $A \in \mc S_n^+$ be irreducible and nonsingular.  Note that $L \in \mc L_+(A)$ is irreducible only if $L\simP C_n$. 
In this case, $A$ has to be a circulant matrix by Corollary \ref{cor:A,L simple eigs}.

\begin{example}\label{ex:small limits}
In this example, we consider different possibilities for $L \in \mc L_+(A)$.  In all cases, we assume that $A$ is a nonsingular irreducible stochastic matrix of the same size as $L$.

For $C_3 \in \mc L_+(A)$, there are two possible multiplicity rearrangement matrices for $(A,C_3)$: 
$$M_1=\left(
\begin{array}{c|ccc}
 & 1 & e^{\frac{2 \pi i}{3}}& e^{\frac{4 \pi i}{3}} \\
\hline
1 & 1 & 0 & 0 \\
\lambda & 0 & 1 & 0 \\
 \bar{\lambda} & 0 & 0 & 1 \\
\end{array}
\right) \text{ and } M_2=\left(
\begin{array}{c|ccc}
  & 1 & e^{\frac{2 \pi i}{3}}& e^{\frac{4 \pi i}{3}} \\
  \hline
 1 & 1 & 0 & 0 \\
 \lambda & 0 & 1 & 1 \\
\end{array}
\right).$$
Note that in the second case, $A$ has a repeated real eigenvalue. In both cases, $A$ has to be a circulant matrix.

For $C_4 \in \mc L_+(A)$ possible multiplicity rearrangement matrices for $(A,C_4)$ are
\begin{align*}
M_1=\left(
\begin{array}{c|cccc}
 & 1 & -1 & i & -i \\
\hline
1 & 1 & 0 & 0 & 0 \\
\lambda_1 & 0 & 1 & 0 & 0 \\
\lambda_2 & 0 & 0 & 1 & 0 \\
 \bar{\lambda}_2 & 0 & 0 & 0 & 1 \\
\end{array}
\right), \, M_2&=\left(
\begin{array}{c|cccc}
 & 1 & -1 & i & -i \\
\hline
1 & 1 & 0 & 0 & 0 \\
\lambda_1 & 0 & 1 & 0 & 0 \\
\lambda_2 & 0 & 0 & 1 & 1 \\
\end{array}
\right) \\
M_3&=\left(
\begin{array}{c|cccc}
 & 1 & -1 & i & -i \\
\hline
1 & 1 & 0 & 0 & 0 \\
\lambda_1 & 0 & 1 & 1 & 1 \\
\end{array}
\right) 
\end{align*}
For $C_2 \oplus C_2 \in \mc L_+(A)$ possible multiplicity rearrangement matrices for $(A,C_2 \oplus C_2)$ are:
\begin{align*}
M_1=\left(
\begin{array}{c|cc}
 & 1 & -1  \\
\hline
1 & 1 & 0   \\
\lambda_1 & 1 & 0   \\
\lambda_2 & 0 & 1  \\
 \lambda_3 & 0 & 1  \\
\end{array}
\right),\, M_1'=\left(
\begin{array}{c|cc}
 & 1 & -1  \\
\hline
1 & 1 & 0   \\
\lambda_1 & 1 & 0   \\
\lambda_2 & 0 & 1  \\
 \bar{\lambda}_2 & 0 & 1  \\
\end{array}
\right), \\
M_2=\left(
\begin{array}{c|cc}
 & 1 & -1  \\
\hline
1 & 1 & 0   \\
\lambda_1 & 1 & 0   \\
\lambda_2 & 0 & 2  \\
\end{array}
\right), \, M_3=\left(
\begin{array}{c|cc}
 & 1 & -1  \\
\hline
1 & 1 & 0   \\
\lambda_1 & 1 & 2   \\
\end{array}
\right)
\end{align*}
Note that $M_1$ and $M_1'$ are the same, but they are listed separately (with the labeling of rows and columns) to emphasize, that they can correspond to a matrix $A$ with real eigenvalues ($M_1$) or to a matrix $A$ that contains a complex conjugate pair of eigenvalues ($M_1'$). Although in this case the eigespaces of the matrix $A$ are not determined by $L$, we can deduce from Theorem \ref{prop:M(A,L)complex}  that for $M_1$, $M_1'$ and $M_2$ we have $\mc E(\lambda_1,A) \subseteq \mc E(1, C_2 \oplus C_2)$. This implies:
$\mc E(\lambda_1,A)=\sspan\left(
\begin{array}{cccc}
 1 & 1 & a & a \\
\end{array}
\right)^T$ for some $a<0$. Similarly, for $\lambda \in \sigma_*(A)\setminus\{1,\lambda_1\}$ we have
 $$\mc E(\lambda,A)\subseteq \mc E(-1,C_2 \oplus C_2)=\sspan\left(
\begin{array}{cccc}
 1 & -1 & 0 & 0 \\
 0 & 0 & 1 & -1 \\
\end{array}
\right)^T.$$
\end{example}

A straightforward observation in our next proposition supports Theorem \ref{prop:LA}.

\begin{proposition}
  Let $L\simP \bigoplus_{i=1}^t C_{n_i}$, $\sum_{i=1}^t n_i=n,$ and let  $\ell_0$ be the least common multiple of $n_1,\ldots,n_t$ ($\ell_0:=\lcm\{n_i; i=1,\ldots, t\}$). Then, $L^{c}=I_n$ if and only if $c=\ell_0 m$ with $m \in \N$. 
\end{proposition}

Given $L=\oplus_{i=1}^t C_{n_i}$, our next result introduces a class of matrices $A \in \mc S_n^+$ with $L \in \mc L_+(A)$. 

\begin{theorem}\label{prop:LA}
Let $n_i \in \N$, $i=1, \ldots, t$, $n:=\sum_{i=1}^t n_i$, and $\ell_0:=\lcm\{n_i; i=1,\ldots, t\}.$ Furthermore, let $L= \oplus_{i=1}^tC_{n_i}$ and $A_0 \in \mc S _n^+$ be a nonsingular infinitely divisible matrix that commutes with $L$. Then $A=LA_0 \in \mc S_n^+$ satisfies $L \in \mc L_+(A)$ and  $\mathcal{CP}(\ell_0)\subseteq \mc P_+(A)$.
\end{theorem}

\begin{proof}
Using ~\cite[Proposition ~7]{kingman} , $A_0=e^{Q}$ for some matrix $Q$ with nonnegative off-diagonal elements and row sums equal to $0$. Hence, $A:=LA_0=Le^{Q}$. From $A_0L=LA_0$ and $L^{\ell_0}=I_n$ we get:
$$\left(L e^{\frac{Q}{k \ell_0+1}}\right)^{k\ell_0+1}=L^{k\ell_0+1}e^{Q}=Le^{Q}=A.$$
Since $L e^{\frac{Q}{k \ell_0+1}} \in \mc S_n^+$, this implies $\left(Le^{\frac{Q}{k\ell_0+1}}\right)_{k=0}^{\infty}\subseteq \mc R_+(A)$ and $(\ell_0 \N+1)\subseteq \mc P_+(A).$ From
$$\lim_{k\rightarrow \infty} L e^{\frac{Q}{k\ell_0+1}}=L \lim_{k\rightarrow \infty}e^{\frac{Q}{k\ell_0+1}}=L.$$
we deduce $L\in \mc L_+(A).$

Finally, if $x \in \mc{CP}(\ell_0)$, then there exists $y \in\N$ so that $x\cdot y \in \ell_0\N+1$. Hence, $x \in \mc P_+(A)$ by Lemma \ref{lem:P+_basic}. 
\end{proof}

\begin{example}\label{ex:M-matrix}
Let
$$M=\left(
\begin{array}{ccc}
 3 & -1 & -1 \\
 -2 & 6 & -3 \\
 -2 & -3 & 6 \\
\end{array}
\right),\, L= \left(
\begin{array}{ccc}
 1 & 0 & 0 \\
 0 & 0 & 1 \\
 0 & 1 & 0 \\
\end{array}
\right),$$ and $$A=LM^{-1}=\frac{1}{45}\left(
\begin{array}{ccc}
 27 & 9 & 9 \\
 18 & 11 & 16 \\
 18 & 16 & 11 \\
\end{array}
\right).$$

Note that $M$ is a nonnsingular $M$-matrix, hence $M^{-1}$ is an infinitely divisible matrix by ~\cite[Theorem~3.6]{highamlin}. A straightforward computation shows that $M$ commutes with $L$. Using Theorem \ref{prop:LA} we deduce: $2 \N+1 \subseteq \mc P_+(A)$ and $L \in \mc L_+(A)$. Since $\sigma(A)=\{1,\frac{1}{5},-\frac{1}{9}\}$, and $A$ has a negative eigenvalue, we conclude that   $\mc P_+(A)=2 \N+1$ and $\mc L_+(A)=\{L\}$. 
\end{example}

\section{Stochastic roots of $2\times 2$ stochastic matrices}\label{sec:2x2}

Stochastic roots of $2 \times 2$ stochastic matrices have been considered by several authors, \cite{tamhuang}, \cite{hegunn}, \cite{linthesis}, \cite{guerry2x2}, and the results in this short section can already be found in the literature. Here we restate known results using notions developed in Section \ref{sec:L_+(A)}, and include proofs for the convenience of the reader. 

\begin{proposition}
(\cite{hegunn}, \cite{linthesis})\label{prop:hegunn}
The matrix $A=\left(
\begin{array}{cc}
 t & 1-t \\
 1-s & s \\
\end{array}
\right)\in \mc S_2^+$,  $s,t \in [0,1)$, has eigenvalues $1$ and $\lambda=s+t-1$. The following statements determine $\mc P_+(A)$: 
\begin{enumerate}
\item $\mc P_+(A)=\N$ if and only if $\lambda=s+t-1 \geq 0$. 

\item $\mc P_+(A)=\{1,3,\ldots,2i+1\}$ for $i \in \mathbb{N}$ if and only if $-1<s+t-1<0$, $s\ne t$ and $i$ is the largest natural number such that $2i+1\leq b$, where
\begin{align}\label{eq:oddrootbound}
   b:=\frac{|\log|s+t-1||}{|\log(1-s)-\log(1-t)|}, \text{    }s,t \in [0,1).
\end{align}
\item $\mc P_+(A)= 2\N_0+1$ if and only if $s=t$ and $2 s<1$.
\end{enumerate}
\end{proposition}

\begin{proof}
Let $s, t \in [0,1)$ and $A$ as defined in the proposition. To determine all possible roots of the matrix $A$ given in the statement, we first find its diagonalization $A=TDT^{-1}$, where 
\begin{align*}
T:=~\left(
\begin{array}{cc}
 1 & -\frac{t-1}{s-1} \\
 1 & 1 \\
\end{array}
\right) \text{ and }
D:=~\left(
\begin{array}{cc}
 1 & 0 \\
 0 & \lambda \\
\end{array}
\right).
\end{align*}
Now, $A=B(\mu)^c$ if and only if 
\begin{align}\label{eq:2x2root}
\begin{split}
B(\mu)&:=~T\left(
\begin{array}{cc}
 1 & 0 \\
 0 & \mu \\
\end{array}
\right) T^{-1}\\
&=\frac{1}{2-s-t}\left( \begin{array}{cc}
	(1-s)+(1-t)\mu & (1-t)(1-\mu) \\
	(1-s)(1-\mu) & (1-t)+(1-s)\mu \\
\end{array}
\right)
\end{split}
\end{align}
for some $\mu \in \C$ satisfying $\mu^c=\lambda$. Since we want $B(\mu)$ to have only real entries, we get the following possibilities for $\mu$:
\begin{enumerate}
\item $\lambda\geq 0$ and $c$ odd: $\mu=\lambda^{\sfrac{1}{c}}\geq 0$,
\item $\lambda \geq 0$ and $c$ even: $\mu=\lambda^{1/c}\geq 0$ or $\mu=-\lambda^{1/c}\geq 0$, 
\item $\lambda<0$, $c$ odd: $\mu=-|\lambda|^{1/c}<0$. 
\end{enumerate}
In all cases $|\mu|<1$, hence the off-diagonal entries of $B(\mu)$ are nonnegative. Moreover, when $\mu>0$, the diagonal entries of $B(\mu)$ are nonnegative. In particular, this proves that $\mc P_+(A)=\N$, when $\lambda \geq 0$.  Nonnegativity of $B(\mu)$ needs to be checked when $\mu=-|\lambda|^{\sfrac{1}{c}}<0$ in the items 2. and 3. above. In this case, $B(-|\lambda|^{\sfrac{1}{c}}) \geq 0$ if and only if: 
$$|\lambda|^{\sfrac{1}{c}}\leq \min\left\{\frac{1-t}{1-s}, \frac{1-s}{1-t}\right\}$$
and hence
$$|\log|\lambda||\leq c \left|\log\left(\frac{1-t}{1-s}\right)\right|.$$
This inequality is always satisfied for $s=t$, and for $s \neq t$ we get the condition:
\begin{equation}\label{eq:2x2c}
     c\leq  \frac{|\log|s+t-1||}{|\log(1-s)-\log(1-t)|}.
 \end{equation}
We conclude that for $s=t$ and $\lambda <0$ we get $\mc P_+(A)=2 \N_0+1$. Finally,  if $s \neq t$, then $\mc P_+(A)=\{1,3,\ldots 2i+1\}$, where $i\in \N$ is the largest possible so that $c=2i+1$ satisfies the equation \eqref{eq:2x2c}. 
\end{proof} 

\begin{remark}
 In Proposition \ref{prop:hegunn} singular matrices $A\in \mc S_2^+$ correspond to the case when $s+t=1$. In this case we have $A=\left(\begin{array}{cc}
          t   & 1-t  \\
           t  & 1-t
        \end{array}\right)$, $A^2=A$, $\mc P_+(A)=\N$  and $\mc L_+(A)=~\left\{A \right\}$.
\end{remark}

\begin{corollary}\label{cor:2x2limits}
 Let $A\in \mc S_2^+$ be as in Proposition \ref{prop:hegunn}. Then:
\begin{enumerate}
\item $\mc L_+(A)=\left\{I_2\right\}$  if and only if $s+t>1$ and $s \neq t$. 
\item $\mc L_+(A)=\left\{I_2,C_2\right\}$ if and only if $s=t$ and $2 s>1$.
\item $\mc L_+(A)=\left\{C_2\right\}$ if and only if  $s=t$ and $s+t<1$.        
        \item $|\mc P_+(A)|<\infty$ if and only if $s+t<1$ and $s\neq t$.  
\end{enumerate}
\end{corollary}

\begin{proof}
Let $A\in \mc S_2^+$ be as given in Proposition \ref{prop:hegunn}. 
Consider first the case when $\lambda>0$. In this case $B(\lambda^{\sfrac{1}{c}}) ^c=A$ and $B(\lambda^{\sfrac{1}{c}}) \in \mc S_2^+$ for all $c \in \N$. Since $\lim_{c\rightarrow \infty}B(\lambda^{\sfrac{1}{c}})=I_2$, we have $I_2 \in \mc L_+(A)$. We also have $B(-\lambda^{\frac{1}{2 c}
})^{2 c}=A$. However, from the proof of Proposition \ref{prop:hegunn} it follows that $B(-\lambda^{\frac{1}{2 c}})\in \mc S_2^+$ for all $c \in \N$ if and only if $s=t$. Hence, $C_2 \in \mc L_+(A)$ if and only if $s=t$. This establishes the first two items of the corollary. 

For $\lambda<0$, Proposition \ref{prop:hegunn} shows that $|\mc P_+(A)|=\infty$ if and only if $s=t$. Moreover, for $s=t$ we have $B(-|\lambda|^{1/(2 c+1)})\in \mc S_2^+$ for all $c \in \N$, and $\lim_{c \rightarrow \infty}B(-|\lambda|^{1/(2 c+1)})=C_2$, proving item 3.

\end{proof}

\section{Circulants: $3\times 3$}\label{sec:3x3}

Let $A\in \mc S_3^+$ with $|\mc P_+(A)|=\infty.$ If $C_3\in \mc L_+(A)$, then $A$ has to be a circulant matrix. However, only some $3 \times 3$ circulant matrices $A\in \mc S_3^+$ satisfy $|\mc P_+(A)|=\infty$. To study this family of matrices we parameterise them with two parameters $s$ and $t$ associated with their eigenvalues:
$\{1,\alpha+i \beta,\alpha-i \beta\}=\{1,e^{-s+it},e^{-s- i t}\}$. Here: $t=\tan^{-1}(\beta/\alpha)$ and $e^{-s}:=\sqrt{\alpha^2+\beta^2}$.  Since $\sqrt{\alpha^2+\beta^2}\leq 1$, we have $s\geq 0$. We will typically take $t \in [-\pi,\pi]$, and will consider $t+2 \pi k$, $k \in \mathbb{Z}$ as equivalent. 
   
Let $$ T=   \left(
\begin{array}{ccc}
 1 & 1 & 1 \\
 1 & e^{-\frac{2 \pi i }{3}} & e^{\frac{2 \pi i }{3}} \\
 1 & e^{\frac{2 \pi i }{3}} & e^{-\frac{2 \pi i }{3}} \\
\end{array}
\right).$$
By Theorem $3.2.3$, \cite{davis_circulant}, any circulant matrix $A \in \mc S_3^+$ can be written as follows:
\begin{align}\label{eq:3x3circulant}
\begin{split}
 \Gamma(s,t)&:=T\diag(1,e^{-s+it},e^{-s-it} )T^{-1}\\
 &=\frac{1}{3}\left(\begin{array}{ccc}
  \gamma_1(s,t)   & \gamma_3(s,t)  & \gamma_2(s,t) \\
   \gamma_2(s,t)   &  \gamma_1(s,t) &  \gamma_3(s,t)\\
      \gamma_3(s,t) &  \gamma_2(s,t)  &  \gamma_1(s,t)\\
      \end{array}\right)
\end{split}
\end{align}
 where:
\begin{align}\label{eq:A(s,t)unique_entries}
    \begin{split}
    \gamma_1(s,t)&:=1+2e^{-s}\sin\left(t+\frac{\pi}{2}\right) \\
    \gamma_2(s,t)&:=1+2e^{-s}\sin\left(t-\frac{\pi}{6}\right)\\
    \gamma_3(s,t)&:=1+2e^{-s}\sin\left(-t-\frac{\pi}{6}\right)
\end{split}
\end{align}
The notation $\Gamma(s,t)$ will be used to denote $3\times 3$ circulant stochastic matrix, as given in equation \eqref{eq:3x3circulant}, throughout this section.

It is clear that $\gamma_i$ are periodic in $t$: $\gamma_i(s,t+2 \pi )=\gamma_i(s,t),$ for $i=1,2,3$, and they satisfy various identities, such as:  
\begin{align}
   \gamma_1(s,t)=\gamma_3(s,t-\frac{2\pi}{3}); \text{ } 
   \gamma_2(s,t)=\gamma_3(s,-t)
   \text { and }
   \gamma_1(s,t)=\gamma_1(s,-t).
\end{align}
Those relations imply some matrix identities as detailed in the lemma below. 

\begin{lemma}\label{lem:A(s,t)permutations}
Matrices $\Gamma(s,t)$ defined in \eqref{eq:3x3circulant} satisfy the following identities:
 \begin{align*}
 &\Gamma(s,-t)=(1 \oplus C_2)\Gamma(s,t)(1 \oplus C_2)\\
 &\Gamma(s,t-\frac{2\pi}{3})=C_3 \Gamma(s,t)\\
 &\Gamma(s,t+\frac{2\pi}{3})=C_3^2 \Gamma(s,t). 
   \end{align*}
\end{lemma}

\begin{proof}
The result can be verified by a direct calculation. 
\end{proof}

To consider the nonnegativity of $\Gamma(s,t)$ we use the following lemma that can be verified by a straightforward calculation. 

\begin{lemma}\label{lem:hatneg}
Let $\hat \gamma(s,t):=1+2 e^{-s}\sin(t)$. Then:
\begin{enumerate}
\item $\hat \gamma(s,t) \geq 0$ for all $t \in [-\pi, \pi]$ if and only if $s \geq \log(2)$. 
\item If $s < \log(2)$, then $\hat \gamma(s,t) < 0$ for $t \in I_{\frac{\pi}{2}}^{-}(s)$, where 
$$I_{\alpha}^{-}(s):=(\alpha-\zeta(s),\alpha+\zeta(s))$$ with $\zeta(s):=\frac{\pi}{2}-\arcsin(\frac{1}{2}e^s).$ 
\end{enumerate}
\end{lemma}

\begin{proposition}\label{prop:3x3_circ_nng}
Let $\Gamma(s,t) \in \mc S_3^+$ be a circulant matrix as defined in \eqref{eq:3x3circulant}, $\psi(s):=-\frac{\pi}{6}+\arcsin\left(\frac{e^s}{2}\right),$ and  $I_{\alpha}^+(s):=[\alpha-\psi(s),\alpha+\psi(s)]$.  
Then:
    \begin{enumerate}
        \item $\Gamma(s,t)$ is nonnegative for all $t \in [-\pi, \pi]$ if and only if $s\geq \log 2.$
        \item For $s < \log(2)$, $\Gamma(s,t)$ is nonnegative for $t \in I_0^+(s) \cup I_{\frac{2 \pi}{3}}^+(s) \cup I_{-\frac{2 \pi}{3}}^+(s)$. 
 \end{enumerate}
\end{proposition}

\begin{proof}
 $\Gamma(s,t)$ is nonnegative if and only if $\gamma_i\geq0$ for all $i \in \{1,2,3\}$. 
The first item of the proposition follows directly from Lemma \ref{lem:hatneg}. For $s<\log(2)$ we have 
$\gamma_1(t)$ is negative on $I_{-\pi}^{-}(s)\sim I_{\pi}^-(s)$, $\gamma_2(t)$ is negative on $I_{-\frac{\pi}{3}}^-(s)\sim I_{\frac{5 \pi}{3}}^-(s),$ and $\gamma_3(t)$ is negative on $I_{\frac{\pi}{3}}^-(s)$. Observation
$$[-\pi,\pi]\setminus (I_{-\pi}^-(s) \cup I_{-\frac{\pi}{3}}^-(s) \cup I_{\frac{\pi}{3}}^-(s))\sim I_0^+(s) \cup I_{\frac{2 \pi}{3}}^+(s) \cup I_{-\frac{2 \pi}{3}}^+(s)$$
completes the proof. 
\end{proof}

When convenient, we will understand the intervals $I_{\alpha}^-(s)$ and $I_{\alpha}^+(s)$ as follows:
\begin{align*}
I_{\alpha}^-(s)&:=\cup_{k \in \mathbb{Z}}(2 
\pi k+\alpha-\zeta(s),2 
\pi k+\alpha+\zeta(s))\\
I_{\alpha}^+(s)&:=\cup_{k \in \mathbb{Z}}[2 
\pi k+\alpha-\psi(s),2 
\pi k+\alpha+\psi(s)].
\end{align*}
In addition, we define $I^+(s):= (I_0^+(s) \cup I_{\frac{2 \pi}{3}}^+(s) \cup I_{-\frac{2 \pi}{3}}^+(s)) \cap (-\pi,\pi]$.  Given $s \geq 0$, we will choose $t \in I^+(s)$ so that $\Gamma(s,t) \in \mc S_3^+$.  

Now let us fix $s_0\geq 0$ and $t_0\in(-\pi,\pi]$, and observe that
 $\Gamma(s_0,t_0)$ has $c$ distinct real $c$-th roots given by:
\begin{align}\label{eq:3x3_circ_root}
 B(c,k)=\Gamma\left(\frac{s_0}{c},\frac{t_0+2\pi k}{c}\right),  
\end{align}
where $k \in \{0,1,\ldots,c-1\}$. Note that the nonnegativity of $B(c,k)$ can be decided directly from Proposition \ref{prop:3x3_circ_nng}.

The discussion so far gives us conditions for $A=\Gamma(s_0,t_0)$ to have nonnegative roots $B(c,k)$ for a given $c\in \N$. To determine the conditions under which $|\mc P_+(A)|=\infty$, we investigate the nonnegativity of $\gamma_i\left (\frac{s_0}{c},\frac{t_0+2\pi k}{c}\right)$ as $c$ goes to infinity.
By Proposition \ref{prop:3x3_circ_nng}, the matrix $B(c,k)=\Gamma\left( \frac{s_0}{c},\frac{t_0+2\pi k}{c} \right)$ is nonnegative if and only if $\frac{t_0+2\pi k}{c}\in ~I_{\alpha}^+\left(\frac{s_0}{c}\right)$ for some $\alpha \in \{0,\frac{2\pi}{3},-\frac{2\pi}{3}\}$. Note that, as $c$ approaches infinity, $\lim_{c\rightarrow \infty}\psi\left(\frac{s}{c}\right)=0$, causing the intervals $I_{\alpha}^+\left(\frac{s_0}{c}\right)$ to converge to the single point $\alpha$. 
 Therefore, in order to assure nonnegativity of $B(c,k)$ we want to choose $k \in \mathbb{Z}$ so that $\frac{t_0+2\pi k}{c}$ as close as possible to $\alpha \in \{0,\frac{2\pi}{3},-\frac{2\pi}{3}\}$ as $c$ grows large. 
 To this end, we define:
\begin{align}\label{eq:k1,k2}
  k_{\lfloor }(\alpha, c):=\left \lfloor \frac{ c \alpha-t_0}{2\pi} \right \rfloor, \,
    k_{\lceil}(\alpha, c):= \left \lceil \frac{c \alpha-t_0}{2\pi} \right \rceil,
\end{align}
where $\alpha \in \{0,\frac{2\pi}{3},-\frac{2\pi}{3}\},$ or equivalently $\alpha \in \{0,\frac{2\pi}{3},\frac{4\pi}{3}\}$.
Note 
\begin{align}\label{eq:limits}
\begin{split}
\lim_{c\rightarrow \infty}B\left(c,k_{\lfloor}\left(0,c\right)\right)&=\lim_{c\rightarrow \infty}B\left(c,k_{\lceil}\left(0,c\right)\right)=I_3\\
\lim_{c\rightarrow \infty}B\left(c,k_{\lfloor}\left(\frac{2\pi}{3},c\right)\right)&=\lim_{c\rightarrow \infty}B\left(c,k_{\lceil}\left(\frac{2\pi}{3},c\right)\right)=C_3^2,\\
\lim_{c\rightarrow \infty}B\left(c,k_{\lfloor}\left(\frac{-2\pi}{3},c\right)\right)&=\lim_{c\rightarrow \infty}B\left(c,k_{\lceil}\left(\frac{-2\pi}{3},c\right)\right)=C_3.
\end{split}
\end{align}

The following lemma allows us to reduce our considerations $t_0 \in [0,\pi]$. 

\begin{lemma}\label{lem:negative}
Let $s_0\geq 0$ and $t_0 \in  I^+(s_0)$. Then 
\begin{enumerate}
\item $\mc P_+(\Gamma(s_0,t_0))=\mc P_+(\Gamma(s_0,-t_0))$
\item  $I_3 \in \mc L_+(\Gamma(s_0,t_0))$ if and only if $I_3 \in \mc L_+(\Gamma(s_0,-t_0))$
\item $C_3 \in \mc L_+(\Gamma(s_0,t_0))$ if and only if $C_3^2 \in \mc L_+(\Gamma(s_0,-t_0))$
\end{enumerate}
\end{lemma}

\begin{proof}
By Lemma \ref{lem:A(s,t)permutations} we have
$\Gamma(s_0,-t_0)=P\Gamma(s_0,t_0)P$
for $P=(1 \oplus C_2)$. Note $P=P^T=P^{-1}$.  
Hence, $\Gamma(s_0,t_0)=B^c$ if and only if $\Gamma(s_0,-t_0)=(PBP^T)^c$, and the first two items of the lemma are established. The last item follows from the equality $PC_3P=C_3^2$. 
\end{proof}

\begin{lemma}\label{lem:function_g}
Let $s\geq 0$, $t \in [0,\pi]$ and $g(z):=1-2 e^{-s z} \sin(\frac{\pi}{6}+t z)$. Then $g(0)=0$ and $g(z)>0$ for all sufficiently small $z>0$ if and only if $s \geq \sqrt{3}t$. 
\end{lemma}

\begin{proof}
The result follows from the power series expansion of $g(z)$ around $z=0$: $$g(z)=z \left(s-\sqrt{3} t\right)+\frac{1}{2} z^2 \left(-s^2+2 \sqrt{3} s t+t^2\right)+O\left(z^3\right)$$ 
\end{proof}

\begin{proposition}\label{prop:circ_roots_around_0}
Let $s_0\geq 0$ and $t_0\in I^+(s_0)$.  Then $I_3 \in \mc L_+(\Gamma(s_0,t_0)
)$ if and only if $ s_0\geq \sqrt{3}|t_0|$. Furthermore, if  $ s_0\geq \sqrt{3}|t_0|$ then $\mc P_+(\Gamma(s_0,t_0))=\N$. 
\end{proposition}

\begin{proof}
By Lemma \ref{lem:negative} it is enough to prove the statement for $t_0 \in [0,\pi]$. Assume $s_0\geq 0$ and $t_0 \in [0,\pi)\cap I^+(s_0)$. 
Note that $k_{\lfloor}(0,c)$ and $k_{\lceil }(0,c)$ are independent of $c$: $k_{\lfloor}(0, c)=-1 $ and $k_{\lceil }(0,c)=0$. By \eqref{eq:limits} we have $I_3 \in \mc L_+(\Gamma(s_0,t_0))$ if and only if at least one of the sequences $B(c,k_{\lfloor}(0, c))$ and $B(c,k_{\lceil}(0, c))$ is nonnegative for all sufficiently large $c \in \N$. 

First let us consider $B(c,0)$. Note that $\gamma_i(\frac{s_0}{c}, \frac{t_0}{c})>0$  for $i=1,2$ and all $s_0 \geq 0$ by Lemma \ref{lem:hatneg}. Substituting $c=\frac{1}{z}$ and using Lemma \ref{lem:function_g} we conclude that $\gamma_3(\frac{s_0}{c}, \frac{t_0}{c})>0$ for all $c \in \N$ if and only if $s_0 \geq \sqrt{3}t_0$. 

Observe that $B(c,-1)=\Gamma(\frac{s_0}{c},\frac{t_0-2 \pi}{c})$ is nonnegative if and only if $\Gamma(\frac{s_0}{c},-\frac{t_0-2 \pi}{c})$ is nonnegative, by Lemma \ref{lem:negative}.  Applying the above argument for $2 \pi -t_0$ we conclude that $B(c,-1)$ is nonnegative for all $c \in \N$ if and only of $s_0 \geq \sqrt{3}(2 \pi -t_0)$. From $t_0 < 2 \pi-t_0$, we conclude that $B(c,-1) \in \mc R_+(\Gamma(s_0,t_0))$ for all $c \in \N$ implies  $B(c,0) \in \mc R_+(\Gamma(s_0,t_0))$ for all $c \in \N$. 
\end{proof}

\begin{remark}\label{rem:two sequences}
Let $\Gamma(s_0,t_0) \in \mc S_3^+$ and $t_0 \in [0,\pi)$. If $s_0\geq \sqrt{3}(2\pi-t_0)$, then $(B(c,0))_{i=1}^{\infty} \subset  \mc S_3^+$ and $(B(c,-1))_{i=1}^{\infty} \subset  \mc S_3^+$. If $\sqrt{3}t_0\leq s_0< \sqrt{3}(2\pi-t_0)$, then we only have $(B(c,0))_{i=1}^{\infty}\subset \mc S_3^+$. 
\end{remark}

\begin{lemma}\label{lem:c=3q}
Let $s_0 \geq 0$, $t_0\in (-\pi,\pi]$, $B(c,k)$ as defined in \eqref{eq:3x3_circ_root}, and $q \in \N$. Then $ B(3q,q)=C_3^2 B(3q,0)$ and $B(3q,-q)=C_3B(3 q,0)$.
\end{lemma}

\begin{proof}
The result follows directly from  Lemma \ref{lem:A(s,t)permutations}.
\end{proof}

\begin{theorem}\label{thm:inf_div_circ}
Let $t_0 \in  I^+(s_0)$, and $s_0 \geq \sqrt{3}|t_0|$. Then $\mc P_+(\Gamma(s_0,t_0))=\N$ and  $\mc L_+(\Gamma(s_0,t_0))=\{I_3,C_3,C_3^2\}.$   
\end{theorem}
\begin{proof}
Under the assumption of the theorem we have $\mc P_+(\Gamma(s_0,t_0))=\N$ and  $I_3 \in \mc L_+(\Gamma(s_0,t_0))$ by Proposition \ref{prop:circ_roots_around_0}. Furthermore, $(B(c,0))_{i=1}^{\infty}\subset \mc S_3^+$ by Remark \ref{rem:two sequences} and  $\lim_{c \rightarrow \infty} B(c,0)=I_3$ by \eqref{eq:limits}.  

By Lemma \ref{lem:c=3q}, we have $(B(3q,q))_{q=1}^{\infty}\subset \mc S_3^+$ and  $(B(3q,-q))_{q=1}^{\infty}\subset \mc S_3^+$. The result follows from $\lim_{q\rightarrow \infty} B(3q,q)=C_3^2$ and $\lim_{q\rightarrow \infty} B(3q,-q)=C_3$.  \end{proof}

Note that Theorem \ref{thm:inf_div_circ} characterises infinitely divisible $3 \times 3$ circulant matrices. Our next result gives a family of arbitrarily finely divisible matrices $A$, with $I_3$ not necessarily an element of $\mc L_+(A)$. 

\begin{corollary}\label{cor:condition_for_C3}
    Let $s_0\geq 0$ and $t_0 \in  I^+(s_0)$. 
    \begin{enumerate}
    \item If 
    $\sqrt{3}|t_0+\frac{2 \pi}{3}| \leq s_0,$
    then $C_3 \in \mc L_+(\Gamma(s_0,t_0))$. 
     \item If 
    $\sqrt{3}|t_0-\frac{2 \pi}{3}| \leq s_0,$
    then $C_3^2 \in \mc L_+(\Gamma(s_0,t_0))$. 
    \end{enumerate}
If either of the conditions above hold, then $\mc{CP}(3)\subseteq \mc P_+(\Gamma(s_0,t_0))$. 
\end{corollary}

\begin{proof}
Let $\mc M_{ID}$ denote the set of infinitely divisible $3 \times 3$ circulant matrices that is by Theorem \ref{thm:inf_div_circ} equal to 
$$\mc M_{ID}=\left\{\Gamma(s,t);  t\in  I^+(s), |t|\leq \frac{\sqrt{3}}{3}s\right\}.$$
Using Lemma \ref{lem:A(s,t)permutations} we deduce
\begin{align*}
C_3\mc M_{ID}=\left\{\Gamma(s,t);  t\in  I^+(s), |t+\frac{2 \pi}{3}|\leq \frac{\sqrt{3}}{3}s\right\}, \\
C_3^2\mc M_{ID}=\left\{\Gamma(s,t);  t\in  I^+(s), |t-\frac{2 \pi}{3}|\leq \frac{\sqrt{3}}{3}s\right\}. 
\end{align*}
Since matrices $\Gamma(s,t)$ commute with $C_3$, Theorem \ref{prop:LA} tells us that $C_3 \in \mc L_+(\Gamma(s,t))$ for every $\Gamma(s,t) \in C_3\mc M_{ID}$. Similarly,   $C_3^2 \in \mc L_+(\Gamma(s,t))$ for every $\Gamma(s,t) \in C_3^2\mc M_{ID}$.
Finally, applying Theorem \ref{prop:LA} for $\ell_0=3$, gives us $\mc{CP}(3) \subseteq \mc P_+(\Gamma(s,t))$ for all $\Gamma(s,t) \in C_3\mc M_{ID} \cup C_3^2\mc M_{ID}$.
\end{proof}

\begin{example}
If the parameters $s_0$ and $t_0$ satisfy: 
$$t_0 \in I_+(s_0) \cap (\frac{\pi}{3},\pi], \,
\sqrt{3}|t_0-\frac{2 \pi}{3}| \leq s_0 < \sqrt{3}t_0,$$
then $I_3\notin \mc L_+(\Gamma(s_0,t_0))$ by Proposition \ref{prop:circ_roots_around_0} and $C_3^2\in \mc L_+(\Gamma(s_0,t_0))$ by Corollary \ref{cor:condition_for_C3}. 
\end{example}

\section{Rank two}\label{sec:rank2}

In this section, we characterise arbitrarily finely divisible matrices with rank two.

\begin{theorem}\label{thm:rank2}
Let $A$ be an $n \times n$ irreducible stochastic matrix with rank two, and nonzero eigenvalues $1$ and $\lambda$. In this case,  $|\mc P_+(A)|=\infty$ if and only if
\begin{align}\label{eq:matrix_rank2}
A \simP \frac{1}{(1+\alpha)} \left(
\begin{array}{cc}
  (\alpha +\lambda){\mathbf{1}}_{n_1} w\trans &  (1-\lambda)\mathbf{1}_{n_1} v\trans \\
  \alpha(1-\lambda)\mathbf{1}_{n_2} w\trans &  (1+\alpha\lambda)\mathbf{1}_{n_2} v\trans  \\
\end{array}
\right),
\end{align}
for some  positive vectors $w \in \R_+^n$, $v\in \R_+^m$  satisfying $w\trans\mathbf{1}_n=v\trans\mathbf{1}_m=1$,  $n_1, n_2 \in \N$ with $n_1+n_2=n$, and either $\lambda>0$ and  $\alpha >0$ or  $\lambda<0$ and  $\alpha=1$. 

Furthermore, if  with $\lambda>0$, then $\mc P_+(A)=\N$, and if $\lambda<0$ and $\alpha=1$, then $\mc P_+(A)=2\N+1$. 
\end{theorem}

\begin{proof}
Let $A\in \mc S_{n}^+$ with $|\mc P_+(A)|=\infty$ have rank $2$ and nonzero eigenvalues $1$ and $\lambda$, where $|\lambda|\leq~ 1$. By Proposition \ref{prop:on limits}, any $L'\in \mc L_+(A)$ has the nonzero spectrum equal to either $\{1,-1\}$ or $\{1,1\}$.  

If $L'$ has nonzero spectrum $\{1,-1\}$, then $L'$ is an irreducible stochastic matrix of the form: 
\begin{align}\label{eq:rank2irrL}
    L'\simP L_{I}:=\left(
\begin{array}{cc}
0_{n_1} & L_{12} \\
 L_{21} & 0_{n_2} \\
\end{array}
\right) \in \mc S_n^+, 
\end{align}
where $L_{12}$ and $L_{21}$ are nonnegative matrices with row sums equal to $1$. Furthermore, the matrix $L_{12}L_{21}$ (and $L_{21}L_{12}$) has an eigenvalue $1$ with multiplicity $1$, and all other eigenvalues equal to $0$. The left and right eigenspaces corresponding to the nonzero eigenvalues of $L_I$ are as follows: 
\begin{align*}
\mc E(1,L_I)=\sspan\left\{\left(
\begin{array}{c}
\mathbf{1}_{n_1} \\
\mathbf{1}_{n_2} 
\end{array}
\right)\right\}, \, \mc E(-1,L_I)=\sspan\left\{\left(
\begin{array}{c}
\mathbf{1}_{n_1} \\
-\mathbf{1}_{n_2} 
\end{array}
\right)\right\}  \\
\mc E_{\ell}(1,L_I)=\sspan\left\{\left(
\begin{array}{c}
w\\
v 
\end{array}
\right)\trans\right\}, \, \mc E_{\ell}(-1,L_I)=\sspan\left\{\left(
\begin{array}{c}
w\\
-v 
\end{array}
\right)\trans\right\}, 
\end{align*}
for some positive vectors $w \in \R_+^{n_1}$, $v \in \R_+^{n_2}$, that satisfy $w\trans \mathbf{1}_{n_1}=v\trans \mathbf{1}_{n_2}=1$.

If $L'$ has nonzero eigenvalues $\{1,1\}$, then $L'\simP L_R=L_{11} \oplus L_{22}$, where $L_{11} \in \mc S_{n_1}^+$ and $L_{22} \in \mc S_{n_2}^+$ are irreducible with the only nonzero eigenvalue equal to $1$.  In this case: 
\begin{align*}
\mc E(1,L_R)&=\sspan\left\{\left(
\begin{array}{c}
\mathbf{1}_{n_1} \\
0_{n_2} 
\end{array}
\right), \, \left(
\begin{array}{c}
0_{n_1} \\
\mathbf{1}_{n_2} 
\end{array}
\right)\right\},  \\
\mc E_{\ell}(1,L_R)&=\sspan\left\{\left(
\begin{array}{c}
w\\
0_{n_2} 
\end{array}
\right)\trans, \left(
\begin{array}{c}
0_{n_1}\\
v 
\end{array}
\right)\trans\right\}, 
\end{align*}
for some positive vectors $w \in \R_+^{n_1}$, $v \in R_+^{n_2}$, normalised so that $w\trans\mathbf{1}_{n_1}=v\trans\mathbf{1}_{n_2}=1$. 
For simplicity of arguments, we will from now on ignore permutational similarity, and assume that either $L'=L_I$ or $L'=L_R$.

Assume $\lambda >0$. If $L_I \in \mc L_+(A)$, then the multiplicity rearrangement matrix for $(A,L_I)$ is equal to $I_2$, and if $L_R \in \mc L_+(A)$, then the multiplicity rearrangement matrix for $(A,L_R)$ is equal to $\npmatrix{1 & 1}\trans$. In the case $\lambda <0$, we can have $L_I \in \mc L_+(A)$ with the multiplicity rearrangement matrix equal to $I_2$. In this case, it is not possible to have $L_R \in \mc L_+(A)$ by Proposition \ref{prop:M(A,L)real}.

Assume now that $L_I \in \mc L_+(A)$. Since in this case, the multiplicity rearrangement matrix is equal to $I_2$, we have:
$\mc E(1,L_I)=\mc E(1,A)$, $\mc E(-1,L_I)=\mc E(\lambda,A)$, $\mc E_{\ell}(1,L_I)=\mc E_{\ell}(1,A)$, $\mc E_{\ell}(-1,L_I)=\mc E_{\ell}(\lambda,A)$. Hence, $A=B(\lambda)$, where we define:
\begin{align*}
B(\mu)&=\frac{1}{2}\left(\begin{array}{cc}
\mathbf{1}_{n_1} & \mathbf{1}_{n_1} \\
\mathbf{1}_{n_2} & -\mathbf{1}_{n_2} 
\end{array}
\right)\left(\begin{array}{cc}
1 & 0 \\
0 & \mu
\end{array}
\right)\left(\begin{array}{cc}
w\trans & v\trans \\
w\trans & -v\trans 
\end{array}
\right)\\
&=\frac{1}{2}\left(\begin{array}{cc}
\mathbf{1}_{n_1}w\trans(1+\mu) & \mathbf{1}_{n_1}v\trans(1-\mu) \\
\mathbf{1}_{n_2}w\trans(1-\mu) & \mathbf{1}_{n_2}v\trans(1+\mu) 
\end{array}
\right).
\end{align*}

Observe, $B(\mu) \in \mc S_{n}^+$ for all $\mu \in [-1,1]$. For $\lambda>0$ we have $A=B(\lambda^{1/c})^c$ and $A=B(-\lambda^{1/2c})^{2c}$  for all $c \in \N$. Furthermore, $\lim_{c \rightarrow \infty}B(\lambda^{1/c})=B(1)$ and $\lim_{c \rightarrow \infty}B(-\lambda^{1/2c})=B(-1)$. Note that $B(-1)$ is of the form $L_I$, while $B(1)$ is of the form $L_R$. Hence, in this case $\mc P_+(A)=\N$ and $\mc L_+(A)$ has at least two elements.  (Note that $\mc L_+(A)$ can also contain elements that do not have rank equal to $2$, see Example \ref{ex:rank2}.)  When $\lambda<0$, we have $A=B(-|\lambda|^{1/(2c+1)})^{2c+1}$, and $\lim_{c\rightarrow \infty}B(-|\lambda|^{1/(2c+1)})=B(-1)$. In this case we have $\mc P_+(A)=2\N+1$ and $B(-1)\in \mc L_+(A)$. 

Now assume $\lambda>0$ and $L_R \in \mc L_+(A)$. The multiplicity rearrangement matrix for $(A,L_R)$ is in equal to 
$$ M=\left(
\begin{array}{c|c}
 & 1   \\
\hline
1 & 1    \\
\lambda & 1    \\
\end{array}
\right).
$$
Hence, $\mc E(1,A)=\sspan\{\mathbf{1}_n\}$, $\mc E(\lambda,A)\subset \mc E(1,L_R)$, and $\mc E_{\ell}(1,A),\mc E_{\ell}(\lambda,A)\subset \mc E_{\ell}(1,L_R)$. 
This implies that there exists $\alpha >0$ so that $A =B(\alpha,\lambda)$, where 
\begin{align*}
B(\alpha,\mu)&:=\frac{1}{1+\alpha}\left(
\begin{array}{cc}
 \mathbf{1}_{n_1} & \mathbf{1}_{n_1} \\
  &   \\
 \mathbf{1}_{n_2} & -\alpha \mathbf{1}_{n_2} \\
\end{array}
\right) \left(
\begin{array}{cc}
1 & 0 \\
0  &  \mu  \\
\end{array}
\right)\left(
\begin{array}{cc}
 \alpha w\trans &  v\trans \\
 w\trans & -v\trans  \\
\end{array}
\right)\\
&=\frac{1}{1+\alpha} \left(
\begin{array}{cc}
  (\alpha +\mu)\mathbf{1}_{n_1} w\trans &  (1-\mu)\mathbf{1}_{n_1} v\trans \\
  \alpha(1-\mu)\mathbf{1}_{n_2}w\trans &  (1+\alpha\mu)\mathbf{1}_{n_2} v\trans  \\
\end{array}
\right).
\end{align*}
Note that $B(\mu)=B(1,\mu)$, and the case with $L_I \in \mc L_+(A)$ we considered above, is for $\lambda>0$ a special case of $B(\alpha,\mu)$. Observe: $B(\alpha,\mu) \in \mc S_2^+$ if and only if $\mu \in \left[\max\{-\alpha, -\frac{1}{\alpha}\},1\right]$. Furthermore,  $A=B(\alpha, \lambda^{1/c})^c$, $B(\alpha, \lambda^{1/c}) \in \mc S_2^+$ for all $c \in \N$, and $\lim_{c \rightarrow \infty}B(\alpha, \lambda^{1/c})=B(\alpha,1)$. We conclude, $\mc P_+(A)=\N$ and $B(\alpha, 1) \in \mc L_+(A)$. (We also have $A=B(\alpha,-\lambda^{1/2c})^{2c}$, but   $B(\alpha,-\lambda^{1/2c}) \in \mc S_2^+$ only for a finite set of $c \in \N$.)
\end{proof}

\begin{example}\label{ex:rank2}
Let $\lambda \in (-1,1)$ and $$A=\frac{1}{2} \left(
\begin{array}{cc}
  (1 +\lambda){\mathbf{1}}_{2} w\trans &  (1-\lambda)\mathbf{1}_{2} v\trans \\
  (1-\lambda)\mathbf{1}_{2} w\trans &  (1+\lambda)\mathbf{1}_{2} v\trans  \\
\end{array}
\right)\in S_4^+.$$  Note that $A$ is of the form 
\eqref{eq:matrix_rank2} for $w\trans=\frac{1}{8}\npmatrix{3 & 5}$, $v\trans=\frac{1}{2}\npmatrix{1 & 1}$ and $\alpha=1$. By Theorem \ref{thm:rank2} and its proof, we have $|\mc P_{+}(A)|=\infty$ and $\mc L_+(A)$ contains a matrix $L_I$ of the form \eqref{eq:rank2irrL}. 
Now Theorem \ref{prop:M(A,L)complex} prescribes the left and right eigenspaces of $L_I$ corresponding to the eigenvalues $1$ and $-1$. Specifically,  $\mc E_{\ell}(1,L_I)=\mc E_{\ell}(1,A)$ implies $\npmatrix{w\trans & v\trans} L_I=\npmatrix{w\trans & v\trans}$. Consequently the matrices $L_{12}$ and $L_{21}$ take the following form:
\begin{align*}
L_{12}=\left(
\begin{array}{cc}
 \beta  & 1-\beta  \\
 \frac{4}{5}-\frac{3 \beta }{5} & \frac{3 \beta }{5}+\frac{1}{5} \\
\end{array}
\right), \, L_{21}=\left(
\begin{array}{cc}
 \gamma  & 1-\gamma  \\
 \frac{3}{4}-\gamma  & \gamma +\frac{1}{4} \\
\end{array}
\right)
\end{align*}
for some $\beta \in [0,1]$ and $\gamma \in [0,\frac{3}{4}]$. In addition, $L_{12}L_{21}$ has to have eigenvalues $1$ and $0$. From $\det(L_{12}L_{21})=\frac{1}{5} (2 \beta -1) (8 \gamma -3)$
we see that either $\beta=\frac{1}{2}$ or $\gamma=\frac{3}{8}$. Equivalently, at least one of the matrices $L_{12}$, $L_{21}$ has to have rank equal to $1$. Let $L_I(\beta,\gamma)$ denote the two parametric family of matrices of the form \eqref{eq:rank2irrL} with blocks $L_{12}$ and $L_{21}$ as above.

Let $\beta=\frac{1}{2}$ and $\lambda>0$. (The case $\lambda<0$ and the case $\gamma=\frac{3}{8}$ can be considered in a similar way.) Note that $L_I\left(\frac{1}{2},\frac{3}{8}\right)$ has rank $2$, and is a limit of roots of $A$ of the form $B(\mu)$ as given in the proof of Theorem \ref{thm:rank2}. To show that all matrices $L_I(\frac{1}{2}, \gamma)$, $\gamma \in [0,\frac{3}{4}]$, belong to $\mc L_+(A)$, we will construct a sequence of matrices in $\mc R_+(A)$ that converges to $L_I(\frac{1}{2}, \gamma)$.  We start by finding an invertible matrix  
$$T(\gamma)=\left(
\begin{array}{cccc}
 1 & 1 & 0 & -\frac{5}{8 \gamma -3} \\
 1 & 1 & 0 & \frac{3}{8 \gamma -3} \\
 1 & -1 & -1 & 0 \\
 1 & -1 & 1 & 0 \\
\end{array}
\right)$$
that satisfies $$(T(\gamma))^{-1}L_I\left(\frac{1}{2}, \gamma\right)T(\gamma)=\left(
\begin{array}{cccc}
 1 & 0 & 0 & 0 \\
 0 & -1 & 0 & 0 \\
 0 & 0 & 0 & 1 \\
 0 & 0 & 0 & 0 \\
\end{array}
\right).$$ 
Let $N_c$, $c \in \N$, be a sequence of  $2 \times 2$ nilpotent matrices, that converges to $J_2(0):=\left(\begin{array}{cc}
   0  & 1 \\
   0  & 0
\end{array}\right)$, and let 
$$B_c(\gamma):=T(\gamma)\left(\left(
\begin{array}{cc}
 1 & 0 \\
 0 & -\lambda^{\frac{1}{2 c}}  \\
\end{array}
\right)
\oplus N_c \right)T(\gamma)^{-1}.$$

We have $B_c^{2 c}=A$ and $\lim_{c \rightarrow \infty}B_c=L_I(\frac{1}{2}, \gamma)$. We still need to choose $N_c$ so that $B_c(\gamma) \in \mc S_4^+$. This can be done in several ways, for example, the following choice of $N_c$ is independent of $\gamma$:
$$N_c:= \frac{(1+\lambda^{\frac{1}{2 c}})}{2}\left(
\begin{array}{cc}
 0 & 1 \\
 0 &  0 \\
\end{array}
\right)$$ and 
results in 
$$B_c(\gamma)=\frac{1}{2}\left(
\begin{array}{cc}
  \left(1-\lambda ^{\frac{1}{2 c}}\right){\mathbf{1}}_{2} w\trans &  \left(1+\lambda ^{\frac{1}{2 c}}\right)\mathbf{1}_{2} v\trans \\\\
  \left(1+\lambda ^{\frac{1}{2 c}}\right)B'(\gamma) &  \left(1-\lambda ^{\frac{1}{2 c}}\right)\mathbf{1}_{2} v\trans  \\
\end{array}
\right),$$ where $B'(\gamma)=\left(
\begin{array}{cc}
 \gamma & 1-\gamma \\
 \frac{3}{4}-\gamma &  \frac{1}{4}+\gamma \\
\end{array}
\right)$.

With this example, we have shown that $\mc L_+(A)$ can contain infinitely many elements. 
\end{example}

The research on arbitrarily finely divisible matrices initiated in this paper can be extended in several directions. For instance, investigations of reducible stochastic matrices and higher-order circulant matrices would be natural continuations of this work. Although not explored in this paper, we believe that the class of arbitrarily finely divisible matrices  has significance in applications, that would provide a motivation for algorithmic approaches to the topic. Once such direction would be finding algorithms that approximate a given stochastic matrix with an arbitrarily finely divisible matrix.

\section*{Acknowledgements}
This publication has emanated from research supported in part by a grant from Science Foundation Ireland under Grant number 18/CRT/6049. For the purpose of Open Access, the author has applied a CC BY public copyright license to any Author Accepted Manuscript version arising from this submission. 

\bibliographystyle{unsrt}
\bibliography{references}

\end{document}